\documentclass[12pt]{article}

\usepackage{amsmath, amssymb, amsfonts}
\usepackage[mathscr]{euscript}
\usepackage{hyperref}
\usepackage{setspace}
\usepackage{geometry}
\usepackage{enumitem}
\usepackage{amsthm}
\usepackage[title]{appendix}

\usepackage{float}

\usepackage{amssymb}
\usepackage[mathscr]{euscript}
\usepackage{amsmath,amsrefs}
\usepackage{stmaryrd}

\usepackage{indentfirst}
\usepackage{empheq}

\newtheorem{theorem}{Theorem}
\newtheorem{corollary}[theorem]{Corollary}
\newtheorem{lemma}[theorem]{Lemma}
\newtheorem{proposition}[theorem]{Proposition}
\newtheorem{definition}[theorem]{Definition}
 
\newtheorem{remark}[theorem]{Remark}

\usepackage{hyperref}
\hypersetup{hypertex = true,
colorlinks = true,
linkcolor = blue,
anchorcolor = blue,
citecolor = blue}

\usepackage{setspace}
\begin{document}

\title{Divisibility of Trace Codes}

\author{
Hexiang Huang\thanks{School of Cyber Science and Technology, Shandong University, Qingdao 266000, China. Email: hexianghuang@foxmail.com} \and
Haihua Deng\thanks{School of Mathematical Sciences, Zhejiang University, Hangzhou 310058, China. Email: haihua.deng@zju.edu.cn (Corresponding author)} \and
Sihuang Hu\thanks{School of Cyber Science and Technology, Shandong University, Qingdao 266000, China. Email: husihuang@sdu.edu.cn}
}
\date{}
\maketitle

\begin{abstract}
A linear code is said to be $\Delta$-divisible if the Hamming weights of all its codewords are divisible by $\Delta$. The $p$-adic valuation of a code is defined as the greatest integer $t$ such that the code is $p^t$-divisible. In this paper, we establish a divisibility criterion for trace codes. Specifically, this criterion provides a systematic method to determine the $p$-adic valuation of the associated trace code, thereby extending Ward's classical divisibility criterion from standard generating sets (or matrices) to generalized generator matrices over an extension field. Furthermore, we present two applications of our framework. The first application provides a concise proof of the celebrated divisibility results on semisimple Abelian codes established by Delsarte and McEliece. The second application establishes several explicit lower bounds on the $p$-adic valuation of the number of solutions over $\mathbb{F}_{q^m}$ (where $q = p^e$) to the Artin--Schreier type equation $ f(x_1,\ldots,x_k)=y^q-y $. In particular, under the coprime condition $\left(d,\frac{q^m-1}{q-1}\right)=1$, we determine the exact minimum $p$-adic valuation of the number of solutions when $f$ is restricted to homogeneous polynomials of degree $d$. 
\end{abstract}


\vspace{0.5em}

\section{Introduction}

The study of the divisibility of the number of zeros for polynomials traces back to the celebrated Chevalley--Warning theorem \cite{chevalley1935, warning1935}, which establishes that the number of common zeros of polynomials $f_1,\ldots,f_r\in\mathbb F_q[x_1,\ldots,x_k]$ in $\mathbb F_q^k$ ($q$ is a power of prime $p$) is divisible by the characteristic $p$ whenever $\sum_{i=1}^r\deg f_i<k$. Classical proofs of this result rely on finite-field power-sum identities. A major strengthening was obtained by Ax \cite{ax1964zeroes}, who, building on Dwork's $p$-adic ideas \cite{dwork1962}, analyzed the associated exponential sums and proved that, for a polynomial $f\in\mathbb F_q[x_1,\ldots,x_k]$ of degree $d$, the number of zeros of $f$ is divisible by $q^{\lceil k/d\rceil-1}$. N. M. Katz \cite[Theorem~1.0]{Katz1971on} extended this result to systems of polynomials. More precisely, the number of common zeros of $f_1,\ldots,f_r$ is divisible by $q^\mu$, where
\[\mu=\max\left\{0,\left\lceil
\frac{k-\sum_{i=1}^r\deg f_i}
{\max_{1\le i\le r}\deg f_i}
\right\rceil
\right\}.
\]
All the above results are sharp, which means that there exist polynomials with the exact divisibility. Subsequent research has refined these estimates for polynomials spanned by some given monomials, employing diverse geometric and combinatorial approaches such as $p$-adic Newton polyhedra \cite{adolphson1987padic} and $p$-weight degree reductions \cite{moreno1995improvement}. Furthermore, a sharp $p$-divisibility is formulated in \cite[Theorem 7]{moreno2004tight} as an optimization problem over a system of modular equations and various techniques are then employed to simplify these constraints and establish concrete bounds.

The divisibility phenomenon revealed by Ax's theorem has found important applications in algebraic coding theory. A linear code $C$ over $\mathbb{F}_q$ is called \textit{$\Delta$-divisible} if $\Delta|\mathrm{wt}(c)$ for every $c\in C$ \cite{ward1981divisible}. Further details on divisible codes can be found in the chapter \cite{Honold2018Partial} by Honold, Kiermaier, and Kurz, and in the extensive survey by Kurz \cite{kurz2021divisible}. In \cite[Theorem 1]{ward1981divisible}, Ward proved a fundamental result regarding divisible codes, which states that any $\Delta$-divisible linear code $C$ over $\mathbb{F}_q$ with $(\Delta,q)=1$ is equivalent to a $\Delta$-fold repetition of some code $C'$, supplemented by sufficiently many 0 coordinates. Therefore, $p$-power divisible codes over $\mathbb{F}_q$ will be of particular interest.

For a nonzero code $C$, the greatest integer $t$ for which $C$ is $p^t$-divisible is called the \emph{$p$-adic valuation} of $C$, denoted by $\nu_p(C)$. Assuming that $q=p^e$, Ax's theorem can be interpreted as a 
$p$-divisibility result for generalized Reed--Muller codes: the $p$-adic valuation of the $d$-th order $k$-variate generalized Reed--Muller code $RM_q(d,k)$ is exactly $\nu_p(RM_q(d,k))=e(\lceil k/d\rceil-1)$. This viewpoint initiated a systematic study of $p$-divisibility in linear codes. In a pioneering work, McEliece \cite[Corollary after Theorem 2]{McEliece1972} gave a spectral characterization of the \(p\)-adic valuation of semisimple cyclic codes over prime fields \(\mathbb{F}_p\), reducing its determination to a minimization problem governed by multiplicative relations among the spectral elements.
Delsarte and McEliece \cite[Theorem 1.1]{delsarte1976} subsequently extended this spectral characterization to semisimple Abelian codes over arbitrary finite fields.

The study of $p$-divisibility was later extended to codes over finite rings. Lower bounds for the $p$-adic valuation of cyclic codes over $\mathbb{Z}/p^c\mathbb{Z}$ were obtained by Calderbank--Li--Poonen \cite[Corollary 3.6 and Theorem 3.7]{CalderbankLiPoonen1997} for $p=2$ and improved by Wilson \cite[Theorems 2--3]{Wilson1995LeeMetric} \cite[Theorem 9]{wilson2006lemma}. D.~J.~Katz further developed a general framework for studying Abelian codes over $\mathbb{Z}/p^c\mathbb{Z}$ and Galois rings using counting polynomials and trace averaging \cite{Katz2005,Katz2006}, and subsequently established a sharp bound for the $p$-adic valuation of Abelian codes over $\mathbb{Z}/p^c\mathbb{Z}$ \cite[Theorem 34]{Katz2008}; see also the correction \cite{Katz2010Correction}. A unified treatment connecting the Delsarte--McEliece and Chevalley--Warning--Ax--Katz viewpoints was subsequently developed in \cite{Katz2012}. On the other hand, Ward \cite[Theorems 4.4 and 5.6]{Ward1983} studied the $p$-divisibility of group codes over finite fields, allowing the underlying group to be non-Abelian, and characterized the $p$-adic valuation in terms of invariant multilinear forms under suitable conditions. De la Cruz and Willems \cite[Theorem 2.2]{deLaCruzWillems2023} later clarified one of these conditions by characterizing it in terms of two-sided ideals of group algebras. Ward \cite[Propositions 6.3--6.4]{Ward1994} also determined the \(p\)-adic valuation of codes arising from powers of the radicals of group algebras of finite \(p\)-groups, thereby extending the study of \(p\)-divisibility to the modular, non-semisimple setting.

A complementary approach, not restricted to cyclic or abelian codes, was initiated by Ward. He developed the combinatorial polarization method \cite{ward1979combinatorial} and, specializing it to the Hamming weight function, obtained a divisibility criterion that can be checked on any additive spanning set of a linear code \cite[Theorem 5.3]{ward1990weight}. The criterion is particularly useful when structural information simplifies the associated moment terms or restricts the feasible exponent tuples. Applied to generalized Reed--Muller codes, Ward's criterion recovers Ax's theorem \cite[Section 6]{ward1990weight}. A further application is to Griesmer codes: the same criterion yields strong $p$-divisibility theorems for codes meeting the Griesmer bound \cite[Theorem 1]{ward1998divisibility}, \cite[Theorem 1.13]{deng2026divisibility}, and the divisibility constraints derived from Ward’s criterion have also contributed to the nonexistence proofs for several putative Griesmer parameter sets, either as a principal tool or in combination with other methods \cite{maruta2004nonexistence,ward2004sequence,kawabata2022nonexistence}.

Ward's divisibility criterion is powerful enough to recover Ax's theorem; however, it does not appear to yield the Delsarte--McEliece theorem for abelian codes in a direct way. It is therefore natural to seek a unified framework that encompasses both results. 

In Section~\ref{Preliminaries}, we recall the necessary preliminaries on Fourier analysis, $p$-adic fields, and Gauss sums. Section~\ref{sec:Divisibility_criterion_of_trace_codes} establishes a generalized divisibility criterion for trace codes by combining the arithmetic of Stickelberger’s theorem with a $p$-adic valuation preserving argument. In Section~\ref{sec:Application_to_Abelian_codes}, we apply this criterion to the trace representation of semisimple abelian codes and recover the Delsarte--McEliece theorem. Our formulation of the $p$-divisibility of abelian codes over finite fields is expressed as a $p$-weight minimization subject to a system of modular equations, and it automatically restricts the search to feasible solutions in a bounded range.

In Section~\ref{sec:artin_schreier}, we investigate the $p$-divisibility of the number of solutions to the Artin--Schreier type equation $f=y^q-y$ in $\mathbb{F}_{q^m}^k$, where $f\in\mathbb{F}_{q^m}[x_1,\ldots,x_k]$. The relation between Artin--Schreier curves and the $p$-divisibility of the corresponding exponential sums has been studied by Blache in \cite{Blache2012}. In \cite{Mueller2013}, Mueller considered the Artin--Schreier equation $f(x)=y^q-y$ with $f\in\mathbb{F}_q[x]$, and obtained a lower bound on the $p$-adic valuation of the number of zeros of this equation in $\mathbb{F}_{q^m}$ for integers $m\ge 2$. Using the trace-code criterion, we formulate a $p$-weight minimization problem that yields a sharp lower bound on this $p$-divisibility for polynomials with restricted degree set (see Section~\ref{sec:artin_schreier} for the definition). We also derive several explicit lower bounds for the same $p$-adic valuation, where Theorem \ref{thm:non_homogeneous_second_bound} includes \cite[Theorem VI.3]{Mueller2013} as a special case. Most notably, when $f$ is a homogeneous polynomial of degree $d$ satisfying a suitable coprimality condition, our method determines a tight lower bound.

Finally, Section~\ref{sec:conclusions} compares our approach with previous methods and concludes the paper.

\section{Preliminaries}\label{Preliminaries}

In this section, we recall some essential preliminaries on Fourier analysis, $p$-adic fields and Gauss sums; we refer the reader to \cite{Babai1989} for Fourier analysis on finite Abelian groups, to \cite{Koblitz1984padicnumber, Robert2000course} for $p$-adic number theory, and to \cite{hou2018lecture} for Stickelberger's theorem on Gauss sums.

Throughout, $\mathbb{F}_q$ denotes the finite field of $q=p^e$ elements, where $p$ is a prime. Tuples are consistently denoted by boldface letters (e.g., $\boldsymbol{r}$), while their components are written in standard font with subscripts (e.g., $r_i$), unless otherwise stated. For $\boldsymbol{r}=(r_1,\ldots,r_k)$, we denote $|\boldsymbol{r}| = \sum_{i = 1}^k r_i$.  For non-negative integers $m < n$, let $[m,n]$ denote the set $\{i\in\mathbb{Z}:m\le i\le n\}$. In particular, when $m=1$, we simply write $[n]$ for $[1,n]$. 

\subsection{Fourier analysis}

Let $A$ be a finite Abelian group, written additively. The \textit{exponent} $N$ of $A$ is the smallest positive integer such that $Na = 0$ for all $a\in A$. Assume that $F$ is a field such that $|A|\neq 0 \in F$ . Further, suppose that $F$ is a field containing a primitive $N$-th root of unity. A function $\chi:A\to F$ is called a \emph{character} of $A$ (over $F$) if $\chi$ is a group homomorphism from $A$ to the multiplicative group $F^\times:=\{x\in F:x\neq 0\}$. 

Clearly, the set of characters of $A$ is closed under pointwise multiplication and thus forms a multiplicative group, denoted by $\widehat{A}$. The \emph{trivial character} of $A$, often denoted by $\chi_0$, is the identity element in the character group $\widehat{A}$, which sends every group element to $1$. It is well known that the character group $\widehat{A}$ is isomorphic to $A$. Denote $\xi_n\in F$ a primitive $n$-th root of unity. In the case of $A = \mathbb{Z}/n\mathbb{Z}$, the character group of $A$ is given by
\[\widehat{A} = \{\chi_a:a\in A\},\]
where $\chi_a:A\to F$ sends $b$ to $\xi_{n}^{ab}$. In general, assume $A = \mathbb{Z}/n_1\mathbb{Z}\times \cdots \times \mathbb{Z}/n_h\mathbb{Z}$. Then the character group of $A$ is given by
\[\widehat{A} = \{\chi_{\boldsymbol{a}}:\boldsymbol{a}\in A\},\]
where $\chi_{\boldsymbol{a}}(\boldsymbol{b}) = \prod_{\ell = 1}^{h}\xi_{n_\ell}^{a_\ell b_\ell}$.

The following result is sometimes called the \emph{first orthogonality relation}.

\begin{lemma}[{\cite[Proposition 1.1]{Babai1989}}]
\label{lem:orthogonality_relation}
    Let $\chi$ be a nontrivial character of $A$. Then 
    \[\sum_{a\in A}\chi(a) = 0.\]
\end{lemma}

It follows from the preceding lemma that
\begin{equation}
\label{eq:orthogonality_of_two_characters}
    \sum_{a\in A}\chi_1(a)\chi_2(a)^{-1} = \begin{cases}
    0, &\text{if }\chi_1\neq \chi_2,\\
    |A|, &\text{otherwise}.
\end{cases}
\end{equation}
Given a function $f:A\to F$. The \emph{Fourier transform} of $f$ is defined as a function $\widehat{f}:\widehat{A}\to F$ which is given by
\[\widehat{f}(\chi) = \sum_{a\in A}f(a)\chi(a)^{-1},\quad\text{for all }\chi\in \widehat{A}.\]
Using \eqref{eq:orthogonality_of_two_characters}, one can verify that the inverse of the Fourier transform is 
\[f(a) = \frac{1}{|A|}\sum_{\chi\in \widehat{A}}\widehat{f}(\chi)\chi(a),\quad\text{for all }a\in A.\]

For $A = \mathbb{Z}/n_1\mathbb{Z}\times \cdots \times \mathbb{Z}/n_h\mathbb{Z}$, each character $\chi_{\boldsymbol{a}}$ is labeled by a group element $\boldsymbol{a}$. For convenience, we will regard the Fourier transform of $f$ as a function on $A$, given by
\begin{equation}
\label{eq:fourier_tranform}
    \widehat{f}(\boldsymbol{a}) = \sum_{\boldsymbol{b}\in A}f(\boldsymbol{b})\chi_{\boldsymbol{a}}(\boldsymbol{b})^{-1},\quad\text{for all }\boldsymbol{a}\in A.
\end{equation}
The corresponding inverse transform is given by
\begin{equation}
\label{eq:inverse_fourier_transform}
    f(\boldsymbol{b}) = \frac{1}{|A|}\sum_{\boldsymbol{a}\in A}\widehat{f}(\boldsymbol{a})\chi_{\boldsymbol{a}}(\boldsymbol{b}),\quad \text{for all }\boldsymbol{b}\in A.
\end{equation}

\subsection{$p$-adic fields}
\label{subsec:p-adic fields}
Let $\mathbb{Z}_p$ be the ring of $p$-adic integers and $\mathbb{Q}_p$ be the field of $p$-adic numbers, equipped with the standard $p$-adic valuation $\nu_p$. Let $\mathbb{C}_p$ denote the $p$-adic completion of the algebraic closure of $\mathbb{Q}_p$. For any positive integer $m$, let $\xi_m \in \mathbb{C}_p$ denote a primitive $m$-th root of unity. It is a standard fact that the ring of integers of the cyclotomic extension $\mathbb{Q}_p(\xi_m)$ is precisely $\mathbb{Z}_p[\xi_m]$.

Consider the specific tower of cyclotomic extensions $\mathbb{Q}_p \subseteq \mathbb{Q}_p(\xi_{q-1}) \subseteq \mathbb{Q}_p(\xi_{p(q-1)})$. The ring of integers $\mathbb{Z}_p[\xi_{q-1}]$ of the intermediate field $\mathbb{Q}_p(\xi_{q-1})$ is a local ring with maximal ideal $(p)$. Its residue field, $\mathbb{Z}_p[\xi_{q-1}]/(p)$, is a finite field of order $q$, which we identify with $\mathbb{F}_q$. Furthermore, $\mathbb{Z}_p[\xi_{q-1}]$ contains the multiplicative group $U_{q-1}$ consisting of all $(q-1)$-st roots of unity. Let $U = U_{q-1} \cup \{0\}$ denote the set of all roots of $X^q-X$.

The canonical map
$$ \varphi: \mathbb{Z}_p[\xi_{q-1}] \to \mathbb{Z}_p[\xi_{q-1}]/(p) = \mathbb{F}_q $$
induces a bijection between $U$ and $\mathbb{F}_q$. Consequently, $U$ serves as a complete set of representatives for $\mathbb{Z}_p[\xi_{q-1}]$ modulo $(p)$. The \textit{Teichm\"uller lift} is defined as the inverse bijection $T: \mathbb{F}_q \to U$, mapping each residue class to its unique representative in $U$. Explicitly, $T(\overline{0}) = 0$ and $T(\overline{\xi_{q-1}^i}) = \xi_{q-1}^i$ for $1 \le i \le q-1$, where $\overline{x} = x + (p) \in \mathbb{F}_q$.

Finally, we consider the subsequent extension to $\mathbb{Q}_p(\xi_{p(q-1)})$. In the corresponding ring of integers $\mathbb{Z}_p[\xi_{p(q-1)}]$, the principal ideal $(p)$ totally ramifies as $(p) = (\pi)^{p-1}$, where $\pi = \xi_p - 1$ is a uniformizing parameter. This allows us to uniquely extend the $p$-adic valuation $\nu_p$ from $\mathbb{Q}_p$ to the entire field $\mathbb{Q}_p(\xi_{p(q-1)})$ by defining $\nu_p(x) = \frac{1}{p-1}\nu_{\pi}(x)$. The valuation $\nu_p(x)$ of any element $x \in \mathbb{Q}_p(\xi_{p(q-1)})$ takes values in $\frac{1}{p-1}\mathbb{Z} \cup \{\infty\}$, with $\nu_p(0) := \infty$.

For $a,b\in\mathbb{N}$, we let $S_p(a)$ denote the sum of the digits in the $p$-adic expansion of $a$ and use $(a,b)$ to denote the greatest common divisor of $a$ and $b$, and $\mathrm{lcm}(a,b)$ to denote the least common multiple of $a$ and $b$. Let $\boldsymbol{x}=(x_1,\ldots,x_n)\in\mathbb{Q}_p(\xi_{p(q-1)})^n$. The $p$-adic valuation of $\boldsymbol{x}$ is defined to be the minimum of $p$-adic valuation of $x_i$'s, i.e., $\nu_p(\boldsymbol{x})=\min\{\nu_p(x_i):1\le i\le n\}$.

\subsection{Gauss sums and Stickelberger's theorem}

A \textit{($p$-adic) character} of $\mathbb{F}_q^{\times}$ is a homomorphism $\chi:\mathbb{F}_q^{\times}\to\mathbb{C}_p^{\times}$. The character group $\widehat{\mathbb{F}_q^\times}$ is generated by the Teichm\"uller lift $T$ when restricted to $\mathbb{F}_q^{\times}$. The \emph{trace} $\mathrm{Tr}_{q/p}$ is defined as a mapping $x\mapsto \sum_{i = 0}^{e-1}x^{p^i}$. The canonical additive character of $\mathbb{F}_q$ is given by $x\mapsto \xi_p^{\mathrm{Tr}_{q/p}(x)}$, where the trace $\mathrm{Tr}_{q/p}$ will be abbreviated as $\mathrm{Tr}$ when there is no ambiguity. Denote the \textit{Gauss sum} over $\mathbb{F}_q$ with respect to $\chi\in \widehat{\mathbb{F}_q^\times}$ by
\[g(\chi) = \sum_{x\in \mathbb{F}_q^\times}\xi_p^{\mathrm{Tr}(x)}\chi(x).\]
For $0 \le i \le q-1$, we define $T^{-i}:\mathbb{F}_q^{\times}\to U$ by $x\mapsto (T^{i}(x))^{-1}$. Clearly, each $T^{-i}$ is a character in $\widehat{\mathbb{F}_q^\times}$. Note that $T^{-0}$ and $T^{-(q-1)}$ are both trivial characters of $\mathbb{F}_q^\times$, and hence $g\left(T^{-0}\right)=g\left(T^{-(q-1)}\right)=-1$. Here, we remark that our convention for Gauss sums differs slightly from some literature, specifically regarding the trivial character. This convention is particularly convenient for our purposes.

The following standard expansion of the canonical additive character will be used in the proof of Theorem \ref{thm:trace_version_criterion_general_case}. We define the coefficients
$$ \lambda_i = \begin{cases}
    1, &\text{if } i = 0,\\
    \frac{1}{q-1}g(T^{-i}), &\text{if } 1\le i\le q-2,\\
    -\frac{q}{q-1}, &\text{if } i = q-1.
\end{cases} $$

\begin{lemma}[{\cite[Lemma 5.1]{hou2018lecture}}]
\label{lem:fourier_expansion_of_additive_character}
    The additive character $x \mapsto \xi_p^{\mathrm{Tr}(x)}$ on $\mathbb{F}_q$ admits a unique polynomial expansion in terms of the Teichm\"uller lift $T(x)$ of degree at most $q-1$:
    $$ \xi_p^{\mathrm{Tr}(x)} = \sum_{i = 0}^{q-1}\lambda_i T(x)^i $$
    for all $x\in \mathbb{F}_q$.
\end{lemma}

\begin{theorem}[Stickelberger's Theorem, cf. {\cite[Theorem 3 in Chapter 14]{ireland1990classical}}]\label{Stickelberger's Theorem}
    For $i = 0,\ldots,q-2$, the $p$-adic valuation of the Gauss sum is given by
    \[\nu_p(g(T^{-i})) = \frac{S_p(i)}{p-1},\]
    where $S_p(i)$ is the digit sum of the base-$p$ expansion of $i$. Therefore, for $i = 0,\ldots,q-1$,
    \[\nu_p(\lambda_i) = \frac{S_p(i)}{p-1}.\]
\end{theorem}

Throughout, when working over \(\mathbb F_{Q}\) ($Q$ is power of $p$), \(T\) denotes the Teichmüller lift \(T:\mathbb F_{Q}\to U_{Q-1}\cup\{0\}\), where \(U_{Q-1}\subset\mathbb C_p^\times\) is the group of \((Q-1)\)-st roots of unity, and \(\lambda_i\), \(0\le i\le Q-1\), denotes the corresponding coefficient in the expansion
$$ \xi_p^{\operatorname{Tr}_{Q/p}(x)} =\sum_{i=0}^{Q-1}\lambda_iT(x)^i. $$
The underlying field will always be clear from the context.

\section{Divisibility criterion of trace codes}\label{sec:Divisibility_criterion_of_trace_codes}

A \textit{linear code} $C$ over $\mathbb{F}_q$ of length $n$ is simply a subspace of $\mathbb{F}_q^n$. A \textit{generator matrix} of $C$ is a matrix whose rows span \(C\); redundant rows are allowed. An \textit{additive spanning set} of $C$ is a set of generators of the additive group $(C,+)$. We first restate  Ward's criterion \cite{ward1990weight} in our notation.

\begin{theorem}[Ward's Divisibility Criterion, {cf. \cite[Theorem 5.3]{ward1990weight}}]
\label{thm:ward_criterion}
Let $C$ be a linear code of length $n$ over $\mathbb{F}_q$ ($q=p^e$). Let $S = \{b_1,\ldots,b_k\}$ be an additive spanning set of $C$ with $b_{ij}$ denoting the $j$-th coordinate of $b_i$. Then the $p$-adic valuation of $C$ is the minimum value of 
$$\frac{1}{p-1}\sum_{i=1}^k S_p(r_i)+\nu_p\left(\sum_{j = 1}^n\prod_{i = 1}^{k}T(b_{ij})^{r_i}\right)-e,$$
where $\boldsymbol{r} \in [0, q-1]^k$ with $|\boldsymbol{r}| \equiv 0 \pmod{q-1}$ and $|\boldsymbol{r}| > 0$.
\end{theorem}

\begin{remark}
We emphasize that the criterion concerns $\mathbb F_q$-linear codes. The use of an additive spanning set does not by itself imply an analogous statement for arbitrary additive
codes that are only $\mathbb F_p$-linear.
\end{remark}

The formulation above uses a fixed finite additive spanning set
$S=\{b_1,\ldots,b_k\}$ and assigns an exponent
$r_i\in[0,q-1]$ to each element of $S$. Thus, a spanning set of cardinality $k$ leads to a naive search space of size at most $q^k$. The size of this search space therefore depends on the chosen spanning set.

If $C$ has dimension $k$ over $\mathbb F_q$, one may start instead from an $\mathbb F_q$-basis $\{a_1,\ldots,a_k\}$. Although the union of its nonzero scalar multiples
\[
\bigcup_{i=1}^k\{\lambda a_i:\lambda\in\mathbb F_q^\times\}
\]
is an additive spanning set and would lead directly to
$k(q-1)$ exponent variables in the preceding formulation, this
scalar redundancy is unnecessary (see Corollary \ref{cor:generator_version}).

We next pass to trace codes, for which the natural representation is given by a generator matrix over an extension field rather than an additive spanning set
over the base field.

\begin{definition}
Let $\mathscr{C}\subseteq\mathbb F_{q^m}^n$ be a $\mathbb{F}_{q^m}$-linear code. Its trace code over
$\mathbb F_q$ is
\[
C = \mathrm{Tr}_{q^m/q}(\mathscr{C})
:=
\left\{
\bigl(\mathrm{Tr}_{q^m/q}(c_1),\ldots,\mathrm{Tr}_{q^m/q}(c_n)\bigr):
(c_1,\ldots,c_n)\in \mathscr{C}
\right\}.
\]
A generator matrix of $\mathscr{C}$ will be called a
\emph{generalized generator matrix} of $C$ over $\mathbb F_{q^m}$.
\end{definition}

Although Ward's criterion is powerful enough to deduce Ax's theorem, it does not seem to directly recover the Delsarte--McEliece theorem for Abelian codes. The primary obstacle is that the formulation of Ward's criterion does not directly apply to generalized generator matrices. Suppose $C$ is an $\mathbb{F}_q$-linear code and $\mathcal{G} = (g_{ij})_{i\in [k],j\in [n]}$ is a generalized generator matrix of $C$ over $\mathbb{F}_{q^m}$. Let $\{\beta_1,\ldots,\beta_{m}\}$ be an $\mathbb{F}_q$-basis of $\mathbb{F}_{q^m}$. Then it can be verified that the matrix
$$G = (\mathrm{Tr}_{q^m/q}(\beta_\ell g_{ij}))_{(i,\ell),j},$$with rows indexed by $(i,\ell)\in [k]\times [m]$ and columns indexed by $j\in [n]$, is a generator matrix of $C$. Applying Ward's criterion to $G$, one needs to compute the $p$-adic valuation of an expression involving trace functions, which requires substantial non-trivial simplification. To overcome this difficulty, we establish a novel divisibility criterion that operates directly on generalized generator matrices, thereby significantly extending its applicability.

We first introduce a crucial linear transformation which leaves the $p$-adic valuation invariant.

\begin{lemma}
\label{lem:transform_preserving_p_adic_valuation}
    Let $Q$ be a power of $p$. Denote $M = \left(\prod_{i = 1}^kT(a_i)^{r_i}\right)_{\boldsymbol{r}\in [0,Q-1]^k, \boldsymbol{a}\in \mathbb{F}_Q^k}$. Then for any vectors $\boldsymbol{x},\boldsymbol{y}$ over $\mathbb{Q}_p(\xi_{p(Q-1)})$ with $\boldsymbol{y} = \boldsymbol{x}M$, we have $\nu_p(\boldsymbol{y}) = \nu_p(\boldsymbol{x})$.
\end{lemma}

\begin{proof}
     We first determine the $p$-adic valuation of $\det(M)$. Let $M_0 = (T(a)^r)_{r,a}$ be a $Q \times Q$ matrix over $\mathbb{Z}_p[\xi_{Q-1}]$ whose rows are indexed by $r\in [0,Q-1]$ and columns by $a\in \mathbb{F}_Q$. Then $M = M_0^{\otimes k}$. Observe that
    $$ M_0=\begin{pmatrix}
     1 & 1 & \cdots & 1\\
     0 &  &  & \\
     \vdots &  & Z & \\
     0 &  &  &
    \end{pmatrix}, $$
    where $Z = (T(a)^r)_{r \in [1, Q-1], a \in \mathbb{F}_Q^\times}$. Thus, $\det(M_0) = \det(Z)$. Let $Z^*= (T(a)^{-r})_{a \in \mathbb{F}_Q^\times, r \in [1, Q-1]}$. Then by the orthogonality relation \eqref{eq:orthogonality_of_two_characters}, $Z Z^* = (Q-1)I$. Thus, $\det(Z)\det(Z^*) = (Q-1)^{Q-1}$. Since $p$ does not divide $Q-1$, we have $\nu_p(\det(Z) \det(Z^*)) = 0$. Because the entries of $Z$ and $Z^*$ lie in the ring of integers $\mathbb{Z}_p[\xi_{Q-1}]$, their determinants are integral, forcing $\nu_p(\det(M_0)) = \nu_p(\det(Z)) = 0$. For the $k$-th tensor power, $\det(M) = \det(M_0)^{k Q^{k-1}}$, so $\nu_p(\det(M)) = 0$.

    Next, since all entries of $M$ are in $\mathbb{Z}_p[\xi_{Q-1}]$, $M$ is an integral matrix. The relation $\boldsymbol{y} = \boldsymbol{x}M$ implies that each coordinate of $\boldsymbol{y}$ is a linear combination of the coordinates of $\boldsymbol{x}$ with integral coefficients. By the properties of the $p$-adic valuation, we obtain $\nu_p(\boldsymbol{y}) \ge \nu_p(\boldsymbol{x})$. 
    
    To establish the reverse inequality, multiply both sides of $\boldsymbol{y} = \boldsymbol{x}M$ on the right by the adjugate matrix $\mathrm{adj}(M)$ to obtain $\boldsymbol{x}\cdot\det(M) = \boldsymbol{y}\cdot\mathrm{adj}(M)$. Since $M$ is integral, $\mathrm{adj}(M)$ is also integral, ensuring $\nu_p(\boldsymbol{y}\cdot \mathrm{adj}(M)) \ge \nu_p(\boldsymbol{y})$. Using $\nu_p(\det(M)) = 0$, we conclude
    $$ \nu_p(\boldsymbol{x}) = \nu_p(\boldsymbol{x})+\nu_p(\det(M)) = \nu_p(\boldsymbol{x}\cdot\det(M)) = \nu_p(\boldsymbol{y}\cdot\mathrm{adj}(M)) \ge \nu_p(\boldsymbol{y}). $$
    Therefore, $\nu_p(\boldsymbol{y}) = \nu_p(\boldsymbol{x})$.
\end{proof}

\begin{remark}
    The preceding proof implies a broader principle: an integral matrix $M$  preserves the $p$-adic valuation of vectors provided that $\nu_p(\det(M)) = 0$.
\end{remark}

\begin{theorem}\label{thm:trace_version_criterion_general_case}
Suppose $C\subseteq\mathbb{F}_{q}^n$ is a $\mathbb{F}_q$-linear code with a generalized generator matrix $G = (g_{ij})_{i\in [k],j\in [n]}$ over $\mathbb{F}_{q^m}$. Then the $p$-adic valuation of $C$ is equal to the minimum value of
\begin{equation}
\label{eq:trace_criterion_objective_value}
    \frac{1}{p-1}\sum_{i = 1}^kS_p(r_i)+\nu_p\left(\sum_{j = 1}^n\prod_{i = 1}^kT(g_{ij})^{r_i} \right)-e,
\end{equation}
where $\boldsymbol{r} \in [0, q^m-1]^k$ with $|\boldsymbol{r}| \equiv 0 \pmod{q-1}$ and $|\boldsymbol{r}| > 0$.
\end{theorem}

\begin{remark}
    The proof proceeds by expressing the codeword weight as a sum of canonical additive characters, which expands into a sum involving Gauss sums via Lemma \ref{lem:fourier_expansion_of_additive_character}. Simplifying this expression and applying the $p$-adic valuation-preserving property of the matrix $M$ (Lemma~\ref{lem:transform_preserving_p_adic_valuation}) concludes the argument.
\end{remark}

\begin{proof}
Note that $C$ is the set of $c(\boldsymbol{\alpha})$ for $\boldsymbol{\alpha}\in \mathbb{F}_{q^m}^k$, where
\begin{align*}
c(\boldsymbol{\alpha}) = \begin{bmatrix} \mathrm{Tr}_{q^m/q}\left(\sum_{i = 1}^k\alpha_i g_{i1}\right)& \cdots& \mathrm{Tr}_{q^m/q}\left(\sum_{i = 1}^k\alpha_i g_{in}\right) \end{bmatrix}.
\end{align*}
For $z\in\mathbb F_q^{\times}$, the map
$x\mapsto\xi_p^{\mathrm{Tr}_{q/p}(xz)}$ is a nontrivial additive character. Hence, by the first orthogonality relation (Lemma \ref{lem:orthogonality_relation}), for $z\in \mathbb{F}_q$, we have
\[\sum_{x\in\mathbb F_q}
\xi_p^{\operatorname{Tr}_{q/p}(xz)}
=
\begin{cases}
q,&\text{if }z=0,\\
0,&\text{otherwise}.
\end{cases}\]
Consequently, we have
\begin{equation}\label{eq:weight_formula}
n-\mathrm{wt}(c(\boldsymbol{\alpha})) = \frac{1}{q}\sum_{j = 1}^n\sum_{x\in \mathbb{F}_q}\xi_p^{\mathrm{Tr}_{q/p}(x\mathrm{Tr}_{q^m/q}(\sum_{i = 1}^k\alpha_i g_{ij}))}.    
\end{equation}
The inner sum of RHS in \eqref{eq:weight_formula} can be rewritten as
\begin{align*}
\sum_{x\in \mathbb{F}_q}\xi_p^{\mathrm{Tr}_{q/p}\left(x\mathrm{Tr}_{q^m/q}\left(\sum_{i = 1}^k\alpha_i g_{ij}\right)\right)} &= \sum_{x\in \mathbb{F}_q}\left(\prod_{i = 1}^k\xi_p^{\mathrm{Tr}_{q^m/p}(x\alpha_i g_{ij})}\right)  \quad\qquad\text{by linearity and transitivity of trace}&\\
&= \sum_{x\in \mathbb{F}_q}\left(\prod_{i = 1}^k\sum_{r_i = 0}^{q^m-1}\lambda_{r_i}T(x\alpha_i g_{ij})^{r_i}\right)  \qquad\qquad\qquad\qquad\qquad\text{by Lemma }\ref{lem:fourier_expansion_of_additive_character}&\\
&=\sum_{x\in \mathbb{F}_q}\left(\sum_{\boldsymbol{r}\in [0,q^m-1]^k} \prod_{i=1}^k \left(\lambda_{r_i}T(x\alpha_i g_{ij})^{r_i} \right) \right) \\
&= q+\sum_{\substack{\boldsymbol{r}\in [0,q^m-1]^k\\ \boldsymbol{r}\neq \boldsymbol{0}}}\sum_{x\in \mathbb{F}_q} \prod_{i=1}^k \left(\lambda_{r_i}T(x\alpha_i g_{ij})^{r_i} \right).
\end{align*}
Plugging the above equality into \eqref{eq:weight_formula} yields
\begin{align*}
\mathrm{wt}(c(\boldsymbol{\alpha}))
&=-\frac{1}{q}\sum_{j = 1}^n \sum_{\substack{\boldsymbol{r}\in [0,q^m-1]^k\\ \boldsymbol{r}\neq \boldsymbol{0}}}\sum_{x\in \mathbb{F}_q} \prod_{i=1}^k \left(\lambda_{r_i}T(x\alpha_i g_{ij})^{r_i} \right) \\
&=-\frac{1}{q}\sum_{\substack{\boldsymbol{r}\in [0,q^m-1]^k\\ \boldsymbol{r}\neq \boldsymbol{0}}} \sum_{j = 1}^n\sum_{x\in \mathbb{F}_q}\prod_{i=1}^k \left(\lambda_{r_i}T(x)^{r_i}T(\alpha_i)^{r_i}T(g_{ij})^{r_i}\right)\\
&=-\frac{1}{q}\sum_{\substack{\boldsymbol{r}\in [0,q^m-1]^k\\ \boldsymbol{r}\neq \boldsymbol{0}}} \left(\prod_{i = 1}^k\lambda_{r_i}\right)\left(\prod_{i =1 }^k T(\alpha_i)^{r_i}\right)\sum_{j = 1}^n\sum_{x\in \mathbb{F}_q} \left(\prod_{i=1}^k T(x)^{r_i}\right) \left(\prod_{i =1 }^k T(g_{ij})^{r_i}\right)  \\
&= -\frac{1}{q}\sum_{\substack{\boldsymbol{r}\in [0,q^m-1]^k\\ \boldsymbol{r}\neq \boldsymbol{0}}}\left(\prod_{i = 1}^k\lambda_{r_i}\right) \left(\prod_{i =1 }^kT(\alpha_i)^{r_i}\right) \left(\sum_{j = 1}^n\prod_{i = 1}^kT(g_{ij})^{r_i}\right)\left(\sum_{x\in \mathbb{F}_q}T(x)^{|\boldsymbol{r}|}\right).
\end{align*}
Regarding the term $\sum_{x\in \mathbb{F}_q}T(x)^{|\boldsymbol{r}|}$, we distinguish the following two cases:
\begin{align*}
\sum_{x\in \mathbb{F}_q}T(x)^{|\boldsymbol{r}|} = \begin{cases} 0, & \text{if } |\boldsymbol{r}| \not\equiv 0 \pmod{q-1}, \\ q-1, & \text{if } |\boldsymbol{r}| \equiv 0 \pmod{q-1} \text{ and } |\boldsymbol{r}| > 0. \end{cases}    
\end{align*}
Therefore, the weight of a codeword $c(\boldsymbol{\alpha})$ can be expressed as:
\begin{equation}
\label{eq:weight_formula_2}
    \mathrm{wt}(c(\boldsymbol{\alpha})) = -\frac{q-1}{q}\sum_{\boldsymbol{r}\in \mathcal{R} }\left(\prod_{i = 1}^k\lambda_{r_i}\right) \left(\prod_{i =1 }^kT(\alpha_i)^{r_i}\right) \left(\sum_{j = 1}^n\prod_{i = 1}^kT(g_{ij})^{r_i}\right),
\end{equation}
where the index set is defined by $\mathcal{R} = \{ \boldsymbol{r} \in [0, q^m-1]^k : |\boldsymbol{r}| \equiv 0 \pmod{q-1},|\boldsymbol{r}| > 0 \}$. 

Define two row vectors $\boldsymbol{y} = (y_{\boldsymbol{\alpha}})_{\boldsymbol{\alpha}\in \mathbb{F}_{q^m}^k}$ with $y_{\boldsymbol{\alpha}} = \mathrm{wt}(c(\boldsymbol{\alpha}))$, and $\boldsymbol{x} = (x_{\boldsymbol{r}})_{\boldsymbol{r}\in [0,q^m-1]^k}$ with
\[x_{\boldsymbol{r}} = \begin{cases}
    -\frac{q-1}{q}\left(\prod_{i = 1}^k\lambda_{r_i}\right)\left(\sum_{j = 1}^n\prod_{i = 1}^kT(g_{ij})^{r_i}\right), &\text{if }\boldsymbol{r}\in \mathcal{R},\\
    0,&\text{otherwise}.
\end{cases}\]
Then we have $\boldsymbol{y} = \boldsymbol{x}M$ where $M$ is the matrix appearing in Lemma \ref{lem:transform_preserving_p_adic_valuation} for $Q = q^m$. It follows from that lemma that 
\[
\nu_p(C)=\nu_p(\boldsymbol y)=\nu_p(\boldsymbol x)
       =\min_{\boldsymbol r\in\mathcal R}\nu_p(x_{\boldsymbol r}).
\]
Since $\nu_p\left(-\frac{q-1}{q}\right)=-e$, for $\boldsymbol r\in\mathcal R$, we have
\[
\nu_p(x_{\boldsymbol r})
=
-e+\sum_{i=1}^k\nu_p(\lambda_{r_i})
+\nu_p\left(\sum_{j=1}^n\prod_{i=1}^kT(g_{ij})^{r_i}\right).
\]
By Stickelberger's theorem (Theorem \ref{Stickelberger's Theorem}), the two expression above yields
\[
\nu_p(C)=\min_{\boldsymbol{r}\in\mathcal R}
\left\{
\frac{1}{p-1}\sum_{i=1}^k S_p(r_i)
+
\nu_p\left(
\sum_{j=1}^n\prod_{i=1}^kT(g_{ij})^{r_i}
\right)
-e
\right\}.
\]
Hence, the result follows.
\end{proof}

\begin{remark}
For computational purposes, one may choose a generator matrix with linearly independent rows, thereby minimizing the number of variables in the search over $\boldsymbol r$.
\end{remark}

\begin{corollary}[Generator-matrix version of Ward's criterion]\label{cor:generator_version}
Suppose $C\subseteq\mathbb{F}_{q}^n$ is a linear code with a generator matrix $G = (g_{ij})_{i\in [k],j\in [n]}$. Then the $p$-adic valuation of $C$ is equal to the minimum value of
\[
\frac{1}{p-1}\sum_{i = 1}^kS_p(r_i)+\nu_p\left(\sum_{j = 1}^n\prod_{i = 1}^kT(g_{ij})^{r_i} \right)-e,
\]
where $\boldsymbol{r} \in [0, q-1]^k$ with $|\boldsymbol{r}| \equiv 0 \pmod{q-1}$ and $|\boldsymbol{r}| > 0$.
\end{corollary}

\begin{proof}
Set $m=1$ in Theorem~\ref{thm:trace_version_criterion_general_case}.
Then $\mathrm{Tr}_{q/q}$ is the identity map.
\end{proof}

Ward's criterion follows immediately from the above corollary, since any additive spanning set of an $\mathbb F_q$-linear code also spans the code over $\mathbb F_q$, and hence forms the rows of a possibly redundant generator matrix. We call the expression $\sum_{j = 1}^{n}\prod_{i = 1}^kT(g_{ij})^{r_i}$ in \eqref{eq:trace_criterion_objective_value} the \emph{moment term}, which encodes structural information about the code. The following comparison concerns the size of the parameter space rather than the running time of an algorithm. For a \(k\)-dimensional code, the generator-matrix formulation starts with \(q^k\) exponent tuples in \([0,q-1]^k\). After imposing the conditions
\[
|\boldsymbol r|\equiv0\pmod{q-1}
\qquad\text{and}\qquad
|\boldsymbol r|>0,
\]
the number of remaining tuples is exactly $(q^k-1)/(q-1)$, which coincides with the number of one-dimensional subspaces of \(\mathbb F_q^k\). Thus, for a general code, this formulation does not by itself reduce the number of candidates compared with projective enumeration. Its main advantage is structural and depends on whether the moment term can be simplified. For semisimple Abelian codes, character orthogonality forces the moment term to vanish unless the exponents satisfy a zero-sum constraint, which may substantially reduce the feasible set; see the proof of Theorem~\ref{thm:Abelian_code_new}. For degree-restricted generalized Reed--Muller trace codes, the moment term factors into finite-field power sums and yields an explicit system of congruence conditions; see Theorem~\ref{thm:p_density_program_RM_codes}. In the latter case, the principal benefit is not necessarily a reduction in enumeration complexity, but rather the conversion of the divisibility problem into a structured $p$-weight minimization problem from which theoretical bounds can be derived.

\section{Application to Abelian codes}\label{sec:Application_to_Abelian_codes}

The purpose of this section is to show that the trace-code criterion becomes particularly transparent for semisimple Abelian codes. An Abelian code is determined by its spectrum. Using the inverse Fourier transform with the spectrum, we can  obtain a trace representation of the code. When the criterion is applied to this representation, character orthogonality turns the moment term into the zero-sum condition $\sum_{i=1}^k r_i\boldsymbol{s}^{(i)} = \boldsymbol{0}$. Consequently, the $p$-adic valuation is reduced to a $p$-weight minimization problem.

\subsection{Spectrum of Abelian codes}

This subsection provides a self-contained introduction to the spectrum of Abelian codes, laying the foundation for the trace representation (Theorem \ref{thm:trace_representation_of_Abelian_codes}) established in the subsequent subsection.  For more details, we recommend \cite[Section~2]{delsarte1976}. Readers focused on applications may proceed directly to the next subsection.

We always assume $A$ is a finite Abelian group and write $A = \mathbb{Z}/{n_1\mathbb{Z}} \times \cdots \times \mathbb{Z}/{n_h\mathbb{Z}}$. The \emph{group algebra} $\mathbb{F}_q[A]$ is a vector space spanned by the formal elements $\{y^{\boldsymbol{a}}:\boldsymbol{a}\in A\}$. We define the \emph{convolution}  $y^{\boldsymbol{a}}* y^{\boldsymbol{b}} = y^{\boldsymbol{a}+\boldsymbol{b}}$ and extend it bilinearly to the space $\mathbb{F}_q[A]$. The convolution makes $\mathbb{F}_q[A]$ a commutative algebra with identity. The elements in $\mathbb{F}_q[A]$ are of the form  
\[f = \sum_{\boldsymbol{a}\in A}f(\boldsymbol{a})y^{\boldsymbol{a}}, \quad\text{for }f(\boldsymbol{a})\in \mathbb{F}_q.\]
We treat the element $f$ as a function $\boldsymbol{a}\mapsto f(\boldsymbol{a})$. In this point of view, $\mathbb{F}_q[A]$ is an $\mathbb{F}_q$-function space of $A$.

An \textit{Abelian code} $C$ is an ideal of the group algebra $\mathbb{F}_q[A]$, where a function $f\in C$ is treated as a tuple $(f(\boldsymbol{a}))_{\boldsymbol{a}\in A}$. Abelian codes were initially studied by Berman \cite{berman1967theory,berman1967semisimple} and independently by MacWilliams \cite{macwilliams1970binary}. Specifically, $C$ is called a \textit{cyclic code} if $A$ is a cyclic group. 

In the sequel, we always assume that $(q,|A|) = 1$. Under this assumption, the group algebra $\mathbb{F}_q[A]$ is semisimple, which implies that any Abelian code can be written as a direct sum of irreducible Abelian codes. Since $(q,|A|) = 1$, the mapping $\boldsymbol{a}\mapsto q\boldsymbol{a}$ is a permutation on $A$, and an orbit under this permutation is called a \emph{$q$-cyclotomic coset}. We denote $\Omega_{\boldsymbol{a}}$ the orbit containing $\boldsymbol{a}$. Clearly, $A$ is a disjoint union of $q$-cyclotomic cosets. Recall that the character group of $A$ is the set of $\chi_{\boldsymbol{a}}$ for $\boldsymbol{a}\in A$ where each $\chi_{\boldsymbol{a}}$ is the group homomorphism
\begin{align*}
    \chi_{\boldsymbol{a}}:A\to \mathbb{F}_{q^m}^{\times}, \;\text{ by }\;
    \boldsymbol{b} \mapsto \prod_{\ell = 1}^h\xi_{n_\ell}^{a_{\ell}b_{\ell}}.
\end{align*}
Two useful properties of $\chi_{\boldsymbol{a}}$ are that $\chi_{\boldsymbol{a}_1}(\boldsymbol{b})\chi_{\boldsymbol{a}_2}(\boldsymbol{b}) = \chi_{\boldsymbol{a}_1+\boldsymbol{a}_2}(\boldsymbol{b})$ and $\chi_{q\boldsymbol{a}}(\boldsymbol{b}) = \chi_{\boldsymbol{a}}(\boldsymbol{b})^q$.

Let $N$ be the exponent of $A$. Since the exponent $N$ and the order $n$ share the same prime factors, the condition $(q,|A|) = 1$ ensures that $(q,N) = 1$. Let $m = \mathrm{ord}_{N}(q)$, the smallest positive integer such that $q^m\equiv 1\pmod{N}$. Note that $\mathbb{F}_{q^m}$ contains a primitive $N$-th root of unity. This allows us to consider the Fourier transform of functions $A\to \mathbb{F}_{q^m}$.

An \emph{idempotent} $E\in \mathbb{F}_q[A]$ is an element satisfying $E*E = E$. The following result shows that $\mathbb{F}_q[A]$ is a principal ideal ring.

\begin{lemma}
    For any Abelian code $C\subseteq \mathbb{F}_q[A]$, there is a unique idempotent element $E\in \mathbb{F}_q[A]$ such that $C = \langle E \rangle:= \{f*E:f\in \mathbb{F}_q[A]\}$.
\end{lemma}

\begin{proof}
    Since $(q,|A|) = 1$, by Maschke's theorem (see \cite[Theorem 4.1.1]{etingof2011introduction}), for any ideal $C$, there exists an ideal $C'$ such that $\mathbb{F}_q[A] = C\oplus C'$. Therefore, there exist $E\in C, E'\in C'$ such that $1 = E+E'$. Since $E*E'\in C\cap C'$ and $C\cap C' = \{0\}$, $E*E' = 0$. Thus,
    \[1 = 1^2 = (E+E')^2 = E^2+E'^2.\]
    This implies $E^2 = E,E'^2 = E'$ since $\mathbb{F}_q[A] = C\oplus C'$ is a direct sum. Clearly, $C\supseteq \langle E\rangle$. Any function $f\in C$ satisfies $f*E' = 0$. Then,
    \[f = f*(E+E') = f*E.\]
    Hence, $C = \langle E\rangle$.

    For uniqueness, suppose $C = \langle E_1\rangle = \langle E_2\rangle$ where $E_1,E_2$ are idempotents. For any $f\in C$,  $f = g*E_1$ for some $g\in \mathbb{F}_q[A]$, and thus $f*E_1= g*E_1*E_1 = g*E_1 = f$. Therefore,
    \[E_1 = E_1*E_2 = E_2*E_1 = E_2.\]
    This completes the proof.
\end{proof}
 
 It is straightforward to verify the  \emph{convolution property} of the Fourier transform, which says that $\widehat{f_1*f_2} = \widehat{f_1}\widehat{f_2}$. Let $E$ be the idempotent of an Abelian code $C$. Then $ \widehat{E}:A\to \mathbb{F}_{q^m}$ is a $0/1$-function since $\widehat{E}\widehat{E} = \widehat{E*E} = \widehat{E}$. The support of $\widehat{E}$, defined as $\{\boldsymbol{a}\in A:\widehat{E}(\boldsymbol{a}) \neq 0\}$, is termed the \emph{spectrum} of $C$, denoted by $\mathrm{spec}(C)$.

 A function from $A$ to $\mathbb{F}_{q^m}$ is generally not the Fourier transform of some function in $\mathbb{F}_q[A]$. 

 \begin{lemma}
 \label{lem:q_symmetric}
      A function $A\to \mathbb{F}_{q^m}$ lies in $\mathbb{F}_q[A]$ if and only if the Fourier transform $\widehat{f}$ is \emph{$q$-symmetric}, namely, $\widehat{f}(q\boldsymbol{a}) = \widehat{f}(\boldsymbol{a})^q$ for all $\boldsymbol{a}\in A$.
 \end{lemma}

 \begin{proof}
     Suppose $f\in \mathbb{F}_q[A]$. Then $f(\boldsymbol{a})^q = f(\boldsymbol{a})$ for all $\boldsymbol{a}$. By the definition of the Fourier transform \eqref{eq:fourier_tranform}, for all $\boldsymbol{a}\in A$, we have
     \begin{align*}
         \widehat{f}(q\boldsymbol{a}) &= \sum_{\boldsymbol{b}\in A}f(\boldsymbol{b})\chi_{q\boldsymbol{a}}(\boldsymbol{b})^{-1}\\
         &= \sum_{\boldsymbol{b}\in A}f(\boldsymbol{b})^q\chi_{\boldsymbol{a}}(\boldsymbol{b})^{-q}\\
         &=\left(\sum_{\boldsymbol{b}\in A}f(\boldsymbol{b})\chi_{\boldsymbol{a}}(\boldsymbol{b})^{-1}\right)^q = \widehat{f}(\boldsymbol{a})^q.
     \end{align*}
     The converse follows by the inverse Fourier transform \eqref{eq:inverse_fourier_transform} and a similar argument.
 \end{proof}
 
Consequently, the spectrum of an Abelian code is a union of some $q$-cyclotomic cosets. Conversely, given $S$ a union of $q$-cyclotomic cosets. Let $\mathbf{1}_S$ be the indicator function of $S$, and let $E_S$ be the inverse Fourier transform of $\mathbf{1}_S$. Since $\mathbf{1}_S$ is $q$-symmetric, $E_S\in \mathbb{F}_q[A]$. Applying the convolution property again, we can show that $E_S$ is an idempotent. It can be verified that the following mapping is  a bijection between the set of Abelian codes on $\mathbb{F}_q[A]$ and the collection of unions of $q$-cyclotomic cosets on $A$, say
\begin{align*}
    \{\text{Abelian codes of $\mathbb{F}_q[A]$} \}&\longleftrightarrow \{\text{unions of $q$-cyclotomic cosets of $A$}\},\\
    C &\longmapsto \mathrm{spec}(C),\\
    \langle E_S\rangle &\longmapsfrom S.
\end{align*}

The following lemma characterizes the codewords in an Abelian code in terms of their Fourier transform. We will use it to obtain the trace representation of Abelian codes (Theorem \ref{thm:trace_representation_of_Abelian_codes}).

\begin{lemma}
\label{lem:characterization_of_Abelian_codes}
    Let $C\subseteq \mathbb{F}_q[A]$ be an Abelian code with the spectrum $S$. For any function $f:A\to \mathbb{F}_{q^m}$, $f\in C$ if and only if $\widehat{f}$ is $q$-symmetric with $\mathrm{supp}(\widehat{f})\subseteq S$.
\end{lemma}

\begin{proof}
Assume $\widehat{f}$ is $q$-symmetric with $\mathrm{supp}(\widehat{f})\subseteq S$. Then $f\in \mathbb{F}_q[A]$ by Lemma \ref{lem:q_symmetric}. The condition $\mathrm{supp}(\widehat{f})\subseteq S$ implies $\widehat{f} = \widehat{f}\mathbf{1}_{S}$. Applying the inverse Fourier transform to the both sides, we have $f = f*E_S$ and thus $f\in C$. The converse follows by a similar argument.
\end{proof}

\begin{remark}
    In the work of MacWilliams \cite{macwilliams1970binary}, a \emph{zero} of $C$ is a character $\chi_{-\boldsymbol{a}}$ such that $\widehat{f}(\boldsymbol{a}) = 0$ for all $f\in C$. Using Lemma \ref{lem:characterization_of_Abelian_codes}, it can be verified that zero set of $C$ is $\{\chi_{-\boldsymbol{a}}:\boldsymbol{a}\in A\backslash\mathrm{spec}(C)\}$.
\end{remark}

\subsection{Divisibility result of Abelian codes}

In this subsection, we present an equivalent form of the Delsarte--McEliece theorem by applying Theorem \ref{thm:trace_version_criterion_general_case}. In order to do that, we need the following result, which states that any Abelian code can be represented as a trace code \cite[Theorem 2.4]{guneri2008multidimensional}. A similar result can also be found in \cite[Eq.~(2.12)]{delsarte1976}. For a better understanding, we give a proof using the spectrum. The proof idea originates from \cite[p.~202]{delsarte1976}

\begin{theorem}[{\cite[Theorem 2.4]{guneri2008multidimensional}}]
\label{thm:trace_representation_of_Abelian_codes}
Let $C$ be an Abelian code of $\mathbb{F}_q[A]$. Suppose the spectrum of $C$ is $\cup_{i = 1}^k\Omega_{\boldsymbol{s}^{(i)}}$. Then $C$ has a generalized generator matrix $\mathcal{G}$ over $\mathbb{F}_{q^m}$, given by
    \[\mathcal{G} = \left(\chi_{\boldsymbol{s}^{(i)}}(\boldsymbol{j}) \right)_{i\in [k], \boldsymbol{j}\in A}.\]
Equivalently, $C$ can be represented as
\begin{equation}
    \label{eq_1_trace_representation}
        C = \left\{\left(\mathrm{Tr}_{q^m/q}\left(\sum_{i = 1}^k\alpha_i \chi_{\boldsymbol{s}^{(i)}}(\boldsymbol{j})\right)\right)_{\boldsymbol{j}\in A}: \boldsymbol{\alpha}\in \mathbb{F}_{q^m}^k\right\}.
    \end{equation}
\end{theorem}

\begin{proof}
We first prove the forward inclusion of \eqref{eq_1_trace_representation}. Let $f\in C$. Then by Lemma \ref{lem:characterization_of_Abelian_codes}, $\widehat{f}$ is $q$-symmetric and $\mathrm{supp}(\widehat{f})\subseteq \cup_{i = 1}^k\Omega_{\boldsymbol{s}^{(i)}}$. Suppose $|\Omega_{\boldsymbol{s}^{(i)}}| = m_i$ for all $i$, so $\Omega_{\boldsymbol{s}^{(i)}}=\{\boldsymbol{s}^{(i)},q\boldsymbol{s}^{(i)},\ldots,q^{m_i-1}\boldsymbol{s}^{(i)}\}$. Then by the inverse Fourier transform \eqref{eq:inverse_fourier_transform}, for all $\boldsymbol{j}\in A$, we have
\begin{align}
    f(\boldsymbol{j}) &= \frac{1}{|A|}\sum_{\boldsymbol{a}\in A}\widehat{f}(\boldsymbol{a})\chi_{\boldsymbol{a}}(\boldsymbol{j})\notag\\
    &=\frac{1}{|A|} \sum_{i=1}^k \sum_{t=0}^{m_i-1} \widehat{f}(q^t\boldsymbol{s}^{(i)})\chi_{\boldsymbol{q^t\boldsymbol{s}^{(i)}}}(\boldsymbol{j})&\text{since $\mathrm{supp}(\widehat{f})\subseteq \cup_{i = 1}^k\Omega_{\boldsymbol{s}^{(i)}}$ }\notag\\
    &= \frac{1}{|A|} \sum_{i=1}^k \sum_{t=0}^{m_i-1} \widehat{f}\left(\boldsymbol{s}^{(i)}\right)^{q^t}\chi_{\boldsymbol{\boldsymbol{s}^{(i)}}}(\boldsymbol{j})^{q^t}&\text{since $\widehat{f}$ is $q$-symmetric }\notag\\
    &= \sum_{i=1}^k \sum_{t=0}^{m_i-1}  \left(\frac{1}{|A|}\widehat{f}\left(\boldsymbol{s}^{(i)}\right)\chi_{\boldsymbol{\boldsymbol{s}^{(i)}}}(\boldsymbol{j})\right)^{q^t}\notag\\
    &= \sum_{i=1}^k \mathrm{Tr}_{q^{m_i}/q}\left(\frac{1}{|A|}\widehat{f}\left(\boldsymbol{s}^{(i)}\right)\chi_{\boldsymbol{\boldsymbol{s}^{(i)}}}(\boldsymbol{j})\right)\label{eq:trace_representation_1},
\end{align}
where the last equality follows by the definition of the trace function and the fact that $\widehat{f}\left(\boldsymbol{s}^{(i)}\right),\chi_{\boldsymbol{\boldsymbol{s}^{(i)}}}(\boldsymbol{j})\in \mathbb{F}_{q^{m_i}}$. We show for example that $\widehat{f}\left(\boldsymbol{s}^{(i)}\right)\in \mathbb{F}_{q^{m_i}}$. Note that $q^{m_i}\boldsymbol{s}^{(i)} = \boldsymbol{s}^{(i)}$. Thus,
\begin{align*}
\widehat{f}\left(\boldsymbol{s}^{(i)}\right)&=\widehat{f}\left(q^{m_i}\boldsymbol{s}^{(i)}\right) = \widehat{f}\left(\boldsymbol{s}^{(i)}\right)^{q^{m_i}}.
\end{align*}
Because the trace map $\mathrm{Tr}_{q^m/q^{m_i}}$ is surjective, we can choose some $\alpha_i\in \mathbb{F}_{q^m}$ such that $\frac{1}{|A|}\widehat{f}(\boldsymbol{s}^{(i)}) = \mathrm{Tr}_{q^m/q^{m_i}}(\alpha_i)$. By the linearity and transitivity of the trace, we obtain
\begin{align}
\mathrm{Tr}_{q^{m_i}/q}\left( \frac{1}{|A|}\widehat{f}\left(\boldsymbol{s}^{(i)}\right)\chi_{\boldsymbol{s}^{(i)}}(\boldsymbol{j})  \right)
    &=  \mathrm{Tr}_{q^{m_i}/q}\left( \mathrm{Tr}_{q^m/q^{m_i}}(\alpha_i) \chi_{\boldsymbol{s}^{(i)}}(\boldsymbol{j}) \right)\notag \\
    &= \mathrm{Tr}_{q^{m_i}/q}\left( \mathrm{Tr}_{q^m/q^{m_i}}(\alpha_i\chi_{\boldsymbol{s}^{(i)}}(\boldsymbol{j}))  \right) \notag\\
    &= \mathrm{Tr}_{q^{m}/q}\left(\alpha_i\chi_{\boldsymbol{s}^{(i)}}(\boldsymbol{j}) \right)\notag.
\end{align}
Substituting the above expression into \eqref{eq:trace_representation_1} yields 
\[f(\boldsymbol{j}) = \mathrm{Tr}_{q^m/q}\left(\sum_{i=1}^k\alpha_i \chi_{\boldsymbol{s}^{(i)}}(\boldsymbol{j})\right).\]
This establishes the forward inclusion for \eqref{eq_1_trace_representation}. The reverse inclusion follows by a similar argument.
\end{proof}

We now apply the trace-code criterion to the generalized generator matrix of Abelian codes to obtain their $p$-divisibility result. Our formulation is as a $p$-weight minimization problem.

\begin{theorem}
\label{thm:Abelian_code_new}
    Let $C$ be an Abelian code in $\mathbb{F}_q[A]$ with the spectrum $\cup_{i = 1}^k\Omega_{\boldsymbol{s}^{(i)}}$. Then the $p$-adic valuation of $C$ is $\Lambda(C)/(p-1)-e$, where $\Lambda(C)$ is given by 
    \begin{equation}
    \label{eq:Abelian_code_p_density}
        \begin{aligned}
        \min\quad & \sum_{i=1}^k S_p(r_i),\\
        \text{s.t.}\quad &\begin{cases}
            \boldsymbol{r} \in [0, q^m-1]^k\backslash\{\boldsymbol{0}\},\\
            |\boldsymbol{r}| \equiv 0 \pmod{q-1},\\
            \sum_{i=1}^k r_i\boldsymbol{s}^{(i)} = \boldsymbol{0} \text{ in $A$}.
        \end{cases}
    \end{aligned}
    \end{equation}
\end{theorem}

\begin{proof}
    By Theorem \ref{thm:trace_representation_of_Abelian_codes}, $C$ has a generalized generator matrix over $\mathbb{F}_{q^m}$ which is given by
    \[\mathcal{G} = \left(\chi_{\boldsymbol{s}^{(i)}}(\boldsymbol{j}) \right)_{i\in [k], \boldsymbol{j} \in A}.\]
    Applying Theorem \ref{thm:trace_version_criterion_general_case} to the matrix $\mathcal{G}$ of $C$, we see that the $p$-adic valuation of $C$ is equal to the minimum value of
    \begin{align}
        &\frac{1}{p-1}\sum_{i = 1}^kS_p(r_i)+\nu_p\left(\sum_{\boldsymbol{j} \in A}\prod_{i = 1}^k T\left(\chi_{\boldsymbol{s}^{(i)}}(\boldsymbol{j})\right)^{r_i} \right)-e\notag\\
        =& \frac{1}{p-1}\sum_{i = 1}^kS_p(r_i)+\nu_p\left(\sum_{\boldsymbol{j} \in A} T\left(\prod_{i = 1}^k\chi_{\boldsymbol{s}^{(i)}}(\boldsymbol{j})^{r_i}\right) \right)-e\notag\\
         =& \frac{1}{p-1}\sum_{i = 1}^kS_p(r_i)+\nu_p\left(\sum_{\boldsymbol{j} \in A} T\left(\chi_{\sum_{i = 1}^kr_i\boldsymbol{s}^{(i)}}(\boldsymbol{j})\right) \right)-e\label{eq:Abelian_code_our_theorem_1},
    \end{align}
    where $\boldsymbol{r} \in [0, q^m-1]^k$ with $|\boldsymbol{r}| \equiv 0 \pmod{q-1}$ and $|\boldsymbol{r}| > 0$. Note that the composition $T\circ \chi_{\sum_{i = 1}^kr_i\boldsymbol{s}^{(i)}}$ is a character of $A$ over $\mathbb{C}_p$, and is trivial if and only if $\sum_{i = 1}^kr_i\boldsymbol{s}^{(i)} = \boldsymbol{0}\in A$. Therefore, 
    \[\sum_{\boldsymbol{j} \in A} T\left(\chi_{\sum_{i = 1}^kr_i\boldsymbol{s}^{(i)}}(\boldsymbol{j})\right) =\begin{cases}
             |A|,&\text{if $\sum_{i = 1}^kr_i\boldsymbol{s}^{(i)} = \boldsymbol{0}\in A$ },\\
             0,&\text{otherwise}.
         \end{cases}. \]
Since $\nu_p(0) = \infty$, it suffices to consider the cases where $\sum_{i = 1}^kr_i\boldsymbol{s}^{(i)} = \boldsymbol{0}\in A$. In this case, we can simplify the expression in \eqref{eq:Abelian_code_our_theorem_1} as $\frac{1}{p-1}\sum_{i = 1}^kS_p(r_i)-e$ since $(q,|A|) = 1$ and thus $\nu_p(|A|)$  =0. The proof is now complete.
\end{proof}

\subsection{Delsarte--McEliece theorem}

Theorem \ref{thm:Abelian_code_new} characterizes the $p$-divisibility of Abelian codes via a $p$-weight minimization, whereas the following classical Delsarte--McEliece theorem \cite[Theorem 1.1]{delsarte1976} expresses it through the minimal length of a specific sequence.

\begin{theorem}[Delsarte--McEliece, {\cite[Theorem 1.1]{delsarte1976}}]
\label{thm:delsarte_mceliece_original}
Let $C$ be an Abelian code of $\mathbb{F}_q[A]$. Then the $p$-adic valuation of $C$ is $L(C)/(p-1) - e$, where $L(C)$ is the least positive integer such that there exist elements $(\boldsymbol{a}_1,t_1), \dots, (\boldsymbol{a}_L,t_L) \in \mathrm{spec}(C)\times [0,e-1]$ (repetition allowed) satisfying 
\begin{align}
    \sum_{j=1}^L p^{t_j} &\equiv 0 \pmod{q-1}, \label{eq:dm_1}\\
    \sum_{j=1}^L p^{t_j} \boldsymbol{a}_j &= \boldsymbol{0} \quad \text{in } A. \label{eq:dm_2}
\end{align}
\end{theorem}

The following proposition establishes the equivalence between Theorem \ref{thm:Abelian_code_new} and Theorem \ref{thm:delsarte_mceliece_original}.

\begin{proposition}
\label{prop:equivalence_with_delsarte_mceliece}
Let $C$ be an Abelian code in $\mathbb{F}_q[A]$ with the spectrum $\cup_{i = 1}^k\Omega_{\boldsymbol{s}^{(i)}}$. Then,
\[
\Lambda(C)=L(C).
\]
\end{proposition}

\begin{proof}
Note that $\mathrm{spec}(C) = \bigcup_{i = 1}^k \Omega_{\boldsymbol{s}^{(i)}} = \{q^j \boldsymbol{s}^{(i)} : i \in [k], j \in [0, m-1]\}$ (the size of each orbit might not be $m$).

We first show that $\Lambda(C) \ge L(C)$. Let $\boldsymbol{r} = (r_1, \dots, r_k)$ be a feasible solution achieving the minimum value $\Lambda(C)$ for the objective function $\sum_{i=1}^k S_p(r_i)$ subject to the constraints in \eqref{eq:Abelian_code_p_density}. Write the $p$-adic expansion of $r_i$ as
\[
r_i = \sum_{j = 0}^{m-1} \sum_{\ell = 0}^{e-1} r_{i,j,\ell} p^{je+\ell} = \sum_{j = 0}^{m-1} \sum_{\ell = 0}^{e-1} r_{i,j,\ell} p^{\ell} q^{j},
\]
where $r_{i,j,\ell} \in [0, p-1]$. Construct a multiset $\mathcal{M}$ consisting of elements $(q^j\boldsymbol{s}^{(i)}, \ell) \in \mathrm{spec}(C) \times [0, e-1]$, where $1\le i\le k, 0\le j\le m-1 , 0\le \ell\le e-1$ and each element appears exactly $r_{i,j,\ell}$ times. By the feasibility of $\boldsymbol{r}$, the constraint $|\boldsymbol{r}| \equiv 0 \pmod{q-1}$ gives
\[
\sum_{(\boldsymbol{a}, t) \in \mathcal{M}} p^t = \sum_{i = 1}^k \sum_{j = 0}^{m-1} \sum_{\ell = 0}^{e-1} r_{i,j,\ell} p^{\ell} \equiv \sum_{i = 1}^k \sum_{j = 0}^{m-1} \sum_{\ell = 0}^{e-1} r_{i,j,\ell} p^{\ell} q^j = \sum_{i = 1}^k r_i = |\boldsymbol{r}| \equiv 0 \pmod{q-1},
\]
which verifies \eqref{eq:dm_1}. Moreover, the constraint $\sum_{i=1}^k r_i \boldsymbol{s}^{(i)} = \boldsymbol{0}$ implies
\[
\sum_{(\boldsymbol{a}, t) \in \mathcal{M}} p^t \boldsymbol{a} = \sum_{i = 1}^k \sum_{j = 0}^{m-1} \sum_{\ell = 0}^{e-1} r_{i,j,\ell} p^{\ell} (q^j \boldsymbol{s}^{(i)}) = \sum_{i = 1}^k r_i \boldsymbol{s}^{(i)} = \boldsymbol{0} \quad \text{in } A,
\]
verifying \eqref{eq:dm_2}. Hence, $\mathcal{M}$ is a valid multiset for Theorem \ref{thm:delsarte_mceliece_original}. The cardinality of $\mathcal{M}$ is 
\[\sum_{i=1}^k \sum_{j=0}^{m-1} \sum_{\ell=0}^{e-1} r_{i,j,\ell} = \sum_{i=1}^k S_p(r_i) = \Lambda(C),\]
which yields $L(C) \le \Lambda(C)$.

Conversely, we establish $L(C) \ge \Lambda(C)$. Let $\mathcal{M}$ be an optimal multiset for Theorem \ref{thm:delsarte_mceliece_original} with minimal cardinality $L(C)$. If any multiplicity is $p$ or greater, one can replace $p$ copies of $(q^j\boldsymbol{s}^{(i)}, \ell)$ with a single copy of $(q^j\boldsymbol{s}^{(i)}, \ell+1)$ (or $(q^{j+1}\boldsymbol{s}^{(i)}, 0)$ if $\ell = e-1$). This replacement strictly reduces the cardinality of $\mathcal{M}$ while preserving \eqref{eq:dm_1} and \eqref{eq:dm_2}, contradicting the minimality of $L(C)$. Thus, the multiplicity of any element $(q^j\boldsymbol{s}^{(i)}, \ell)$ in $\mathcal{M}$ cannot exceed $p-1$. Let $r_{i,j,\ell} \in [0, p-1]$ denote the multiplicity of $(q^j\boldsymbol{s}^{(i)}, \ell)$ in the optimal multiset $\mathcal{M}$. Define $r_i = \sum_{j=0}^{m-1} \sum_{\ell=0}^{e-1} r_{i,j,\ell} p^{\ell} q^j$ for $i \in [k]$. Reading the previous equations backwards shows that $\boldsymbol{r} = (r_1, \dots, r_k)$ satisfies the constraints in \eqref{eq:Abelian_code_p_density}. Thus, $\boldsymbol{r}$ is a feasible solution yielding an objective value $\sum_{i=1}^k S_p(r_i) = L(C)$, which gives $\Lambda(C) \le L(C)$.

Combining both inequalities yields $\Lambda(C) = L(C)$.
\end{proof}

\begin{remark}
One advantage of using the trace-code criterion to establish the $p$-divisibility of Abelian codes is that it intrinsically restricts each variable $r_i$ to the finite range $[0, q^m-1]$. This contrasts with the unbounded multiplicities allowed for the spectrum tuples in Theorem \ref{thm:delsarte_mceliece_original}.
\end{remark}

\section{Application to Artin--Schreier type equations}\label{sec:artin_schreier}

A polynomial $f\in \mathbb{F}_{q^m}[x_1,\ldots,x_k]$ is \emph{reduced} if every individual degree of $f$ is at most $q^m-1$. In the sequel, all polynomials are assumed to be reduced.

In this section, we will consider the $p$-adic valuation of $N_{q^m}(f=y^q-y)$, the number of solutions $(x_1,\ldots,x_k,y)\in\mathbb{F}_{q^m}^{k+1}$ to the Artin--Schreier type equation $f(x_1,\ldots,x_k)=y^q-y$ for some $f\in\mathbb{F}_{q^m}[x_1,\ldots,x_k]$, where the subscript of $N_{q^m}$ implies that we consider the solutions over the field $\mathbb{F}_{q^m}$. Note that $\mathrm{Tr}_{q^m/q}(f)=0$ if and only if $f=y^q-y$ for some $y\in\mathbb{F}_{q^m}$ (see \cite[Theorem 2.25]{LidlNiederreiter1997}). Thus,
\begin{equation}\label{solutions equivalence}
N_{q^m}(f=y^q-y)=qN_{q^m}(\mathrm{Tr}_{q^m/q}(f)=0).    
\end{equation}
Then it suffices to consider the number of solutions to the equation $\mathrm{Tr}_{q^m/q}\left(f(x_1,\ldots,x_k)\right)=0$. Denote $c(f)$ the evaluation vector of $\mathrm{Tr}_{q^m/q}(f)$ on $\mathbb{F}_{q^m}^k$. We have 
\[\mathrm{wt}(c(f))=q^{mk}-N_{q^m}(\mathrm{Tr}_{q^m/q}(f)=0).\]
This allows us to apply Theorem \ref{thm:trace_version_criterion_general_case}.  

Suppose $f = f_{d_1}+\cdots+f_{d_t}$, where $f_{d_i}$ is a homogeneous part of degree $d_i$ with $d_1>\cdots>d_t$. We then denote
\begin{equation}\label{eq:definition_of_degree_set}
    \mathrm{Deg}(f) = \{d_1,\ldots,d_t\},
\end{equation}
where we use the convention $\mathrm{Deg}(0)=\varnothing$. Let $D\subseteq [0,k(q^m-1)]$. The \textit{generalized degree-restricted Reed--Muller code} over $\mathbb{F}_{q^m}$ of degree set $D$ in $k$ variables, denoted by $RM_{q^m}(D,k)$, is the space of all $k$-variate polynomials $f$ with $\mathrm{Deg}(f)\subseteq D$ evaluated on $\mathbb{F}_{q^m}^k$. Denote $\mathcal{T}_D$ the set of tuples $\boldsymbol{t}\in [0,q^m-1]^k $ with $|\boldsymbol{t}|\in D$.

\begin{theorem}
\label{thm:p_density_program_RM_codes}
Let $D\subseteq[0,k(q^m-1)]$ with nonempty $D\neq \{0\}$. Let $f\in \mathbb{F}_{q^m}[x_1,\ldots,x_k]$ be a polynomial with $\mathrm{Deg}(f)\subseteq D$. Then the $p$-adic valuation of $N_{q^m}(f=y^q-y)$ is at least
    \begin{equation}
    \label{eq:p_density_program}
    \begin{aligned}
        \min\quad &\frac{1}{p-1}\sum_{\boldsymbol{t}\in \mathcal{T}_D}S_p(r_{\boldsymbol{t}}),\\
        \text{s.t.}\quad &\begin{cases}
            \boldsymbol{r}\in [0,q^m-1]^{\mathcal{T}_D},\\
            |\boldsymbol{r}|\equiv 0\pmod{q-1},\\
            \sum_{\boldsymbol{t}\in \mathcal{T}_D}r_{\boldsymbol{t}}t_i \text{ is a positive multiple of } q^m-1,\ 1\le i\le k.
        \end{cases}
    \end{aligned}
    \end{equation}
 Moreover, there exists an $f$ with $\mathrm{Deg}(f)\subseteq D$ achieving this minimum value.
\end{theorem}

\begin{proof}
    We begin by considering the trace code $C = \mathrm{Tr}_{q^m/q}(RM_{q^m}(D,k))$. A natural generalized generator matrix over $\mathbb{F}_{q^m}$ of $C$ is given by
    \[\mathcal{G} = \left(\prod_{i = 1}^k x_i^{t_i}\right)_{\boldsymbol{t}\in \mathcal{T}_D,\boldsymbol{x}\in\mathbb{F}_{q^m}^k}.\]
    For any polynomial $f \in \mathbb{F}_{q^m}[x_1,\ldots,x_k]$ with $\mathrm{Deg}(f) \subseteq D$, the evaluations of $\mathrm{Tr}_{q^m/q}(f)$ on $\mathbb{F}_{q^m}^k$ form a codeword $c(f) \in C$. 
    By Theorem \ref{thm:trace_version_criterion_general_case}, the minimum $p$-adic valuation of the codeword weights, $\min\nu_p(\mathrm{wt}(c(f)))$, is given by the minimum of
    \begin{align}
    V(\boldsymbol{r}) 
    &= \frac{1}{p-1}\sum_{\boldsymbol{t}\in \mathcal{T}_D}S_p(r_{\boldsymbol{t}})+\nu_p\left(\sum_{(x_1,\ldots,x_k)\in \mathbb{F}_{q^m}^k} \prod_{\boldsymbol{t}\in \mathcal{T}_D} T\left(\prod_{i = 1}^k x_i^{t_i}\right)^{r_{\boldsymbol{t}}} \right)-e \notag\\
    &=\frac{1}{p-1}\sum_{\boldsymbol{t}\in \mathcal{T}_D}S_p(r_{\boldsymbol{t}})+\nu_p\left(\sum_{(x_1,\ldots,x_k)\in \mathbb{F}_{q^m}^k} \prod_{i = 1}^k T(x_i)^{\sum_{\boldsymbol{t}\in \mathcal{T}_D}t_ir_{\boldsymbol{t}}} \right)-e \notag\\
    &=\frac{1}{p-1}\sum_{\boldsymbol{t}\in \mathcal{T}_D}S_p(r_{\boldsymbol{t}})+\nu_p\left(\prod_{i = 1}^k \sum_{x_i\in \mathbb{F}_{q^m}} T(x_i)^{u_i(\boldsymbol{r})} \right)-e \label{eq:objective function}
    \end{align} 
    over all $\boldsymbol{r} \in [0, q^m-1]^{\mathcal{T}_D}$ satisfying $|\boldsymbol{r}| \equiv 0 \pmod{q-1}$ and $|\boldsymbol{r}| > 0$, where $u_i(\boldsymbol{r}):= \sum_{\boldsymbol{t}\in \mathcal{T}_D} t_ir_{\boldsymbol{t}}$.
    
    Observe that the inner sum $\sum_{x\in \mathbb{F}_{q^m}} T(x)^{u}$ evaluates to $q^m$ if $u=0$, to $q^m-1$ if $u > 0$ with $u \equiv 0 \pmod{q^m-1}$, and to $0$ otherwise. Hence, the minimum valuation occurs exactly when $u_i(\boldsymbol{r}) \equiv 0 \pmod{q^m-1}$ for all $1 \le i \le k$. If exactly $\ell$ components of $(u_1(\boldsymbol{r}), \ldots, u_k(\boldsymbol{r}))$ are zero, the product evaluation yields a $p$-adic valuation of $em\ell$ (since $\nu_p(q^m-1)=0$). Thus, the objective function \eqref{eq:objective function} simplifies to
    \[ V(\boldsymbol{r}) = \frac{1}{p-1}\sum_{\boldsymbol{t}\in \mathcal{T}_D}S_p(r_{\boldsymbol{t}})+em\ell-e.\]
    We next demonstrate that the global minimum is always attained when $\ell = 0$. Suppose that $u_i(\boldsymbol{r}) = \sum_{\boldsymbol{t}\in \mathcal{T}_D} t_ir_{\boldsymbol{t}}= 0$ for some $i$, which forces $r_{\boldsymbol{t}} = 0$ for any $\boldsymbol{t} \in \mathcal{T}_D$ with $t_i > 0$ (the existence of $\boldsymbol{t}$ with $t_i>0$ is guaranteed by $D\ne \{0\}$). Let $\boldsymbol{s}\in \mathcal{T}_D$ such that $s_i>0$. We can modify $\boldsymbol{r}$ to $\boldsymbol{r}'$ by setting 
    \[r_{\boldsymbol{t}}'=
    \begin{cases}
    r_{\boldsymbol{t}}, &\text{if } \boldsymbol{t}\ne \boldsymbol{s},\\     
    q^m-1, &\text{if } \boldsymbol{t}= \boldsymbol{s}.
    \end{cases}\]
    This shift preserves both the condition $|\boldsymbol{r}'| =|\boldsymbol{r}|+q^m-1\equiv 0 \pmod{q-1}$ and the congruences $u_j(\boldsymbol{r}') =u_j(\boldsymbol{r})+s_j(q^m-1)\equiv 0 \pmod{q^m-1}$ for all $j\in [k]$. Let $\ell'$ be the number of zeros in the tuple $(u_1(\boldsymbol{r}'),\ldots,u_k(\boldsymbol{r}'))$. Since $u_i(\boldsymbol{r}') = u_i(\boldsymbol{r})+\boldsymbol{s}_i(q^m-1)>0$, we have $\ell'\le \ell-1$. We claim that $V(\boldsymbol{r}')\le V(\boldsymbol{r})$. This is because
    \begin{align*}
        V(\boldsymbol{r}') &=\frac{1}{p-1}\sum_{\boldsymbol{t}\in \mathcal{T}_D}S_p(r_{\boldsymbol{t}}')+em\ell'-e\\
        &=\frac{1}{p-1}\sum_{\boldsymbol{t}\in \mathcal{T}_D}S_p(r_{\boldsymbol{t}})+\frac{S_p(q^m-1)}{p-1}+em\ell'-e\\
        &\le \frac{1}{p-1}\sum_{\boldsymbol{t}\in \mathcal{T}_D}S_p(r_{\boldsymbol{t}})+em+em(\ell-1)-e\\
        &= V(\boldsymbol{r}).
    \end{align*}
 After this replacement, $V(\boldsymbol{r})$ does not increase. By repeating this adjustment, we can eliminate all zero $u_i(\boldsymbol{r})$'s without increasing the overall valuation. Therefore, the minimum value of $\nu_p(\mathrm{wt}(c(f)))$ can be precisely formulated as the following optimization program:
    \begin{equation*}
    \begin{aligned}
       \min\quad &V'(\boldsymbol{r})=\frac{1}{p-1}\sum_{\boldsymbol{t}\in \mathcal{T}_D}S_p(r_{\boldsymbol{t}})-e,\\
        \text{s.t.}\quad &\begin{cases}
            \boldsymbol{r}\in [0,q^m-1]^{\mathcal{T}_D},\\
            |\boldsymbol{r}|\equiv 0\pmod{q-1},\\
            \sum_{\boldsymbol{t}\in \mathcal{T}_D}r_{\boldsymbol{t}}t_i \text{ is a positive multiple of } q^m-1,\ 1\le i\le k.
        \end{cases}
    \end{aligned}
    \end{equation*}
    
    Below, we show that this minimum is strictly bounded away from $emk$. Since $D\neq \{0\}$, there exist distinct tuples $\boldsymbol{t}^{(1)},\ldots,\boldsymbol{t}^{(k')}\in \mathcal{T}_D$ ($1\le k'\le k$) such that the union of their supports covers $[k]$. Consider the specific assignment $\boldsymbol{r}\in [0,q^m-1]^{\mathcal{T}_D}$ where $r_{\boldsymbol{t}} = q^m-1$ if $\boldsymbol{t} = \boldsymbol{t}^{(j)}$ for some $1\le j\le k'$, and $r_{\boldsymbol{t}} = 0$ elsewhere. It is straightforward to verify that $\boldsymbol{r}$ is a feasible solution to the above program. For this configuration, the objective value is at most $ emk-e$, which is strictly less than $emk$. 
    
    Combining this strict upper bound $\min \nu_p(\mathrm{wt}(c(f)))<emk$ with the  identity
    $$N_{q^m}(\mathrm{Tr}_{q^m/q}(f)=0) + \mathrm{wt}(c(f)) = p^{emk},$$
    the properties of non-Archimedean valuations guarantee that
    \[\min\nu_p(N_{q^m}(\mathrm{Tr}_{q^m/q}(f)=0)) = \min \nu_p(\mathrm{wt}(c(f))),\]
    where $f$ ranges over $\mathbb{F}_{q^m}[x_1,\ldots,x_k]$ with $\mathrm{Deg}(f)\subseteq D$. 
    
    Finally, using \eqref{solutions equivalence}, the valuation of the number of solutions to the Artin--Schreier type equation is given by
    \begin{align*}
        \min\nu_p(N_{q^m}(f=y^q-y)) = \nu_p(q) + \min\nu_p(N_{q^m}(\mathrm{Tr}_{q^m/q}(f)=0)) = e+\min V'(\boldsymbol{r}).
    \end{align*}
    This completes the proof.
\end{proof}

\begin{proposition}\label{upper bound}
    For any integer $1\le d\le k(q^m-1)$, there exists a homogeneous polynomial $f\in \mathbb{F}_{q^m}[x_1,\ldots,x_k]$ with degree $d$ such that the $p$-adic valuation of $N_{q^m}(f=y^q-y)$ is at most
    \begin{align*}
        em\left\lceil \frac{k}{d} \right\rceil.    
    \end{align*}
\end{proposition}

\begin{proof}
    We show this by providing some feasible $\boldsymbol{r}$ to the program \eqref{eq:p_density_program} when $D = \{d\}$.

    First, if $d\ge k$, then there exists a tuple $(t_1,\ldots,t_k)$ such that $1\le t_i\le q^m-1$ and $\sum_{i=1}^k t_i=d$. Take such a tuple $(t_1,\ldots,t_k)$ and set
    \begin{align*}
        r_{\boldsymbol{t}}=\begin{cases}
        q^m-1,\text{ if }\boldsymbol{t}=(t_1,\ldots,t_k),\\
        0,\text{ otherwise.}
        \end{cases}    
    \end{align*}
    Since $|\boldsymbol{r}|=\sum_{\boldsymbol{t}\in\mathcal{T}_{\{d\}}}r_{\boldsymbol{t}}=q^m-1\equiv 0\pmod{q-1}$ and $\sum_{\boldsymbol{t}\in\mathcal{T}_{\{d\}}}r_{\boldsymbol{t}}t_i=t_i(q^m-1)$ is a positive multiple of $q^m-1$ for $1\le i\le k$, such $\boldsymbol{r}$ satisfies \eqref{eq:p_density_program}. Thus, by Theorem \ref{thm:p_density_program_RM_codes}, there exists a homogeneous polynomial $f$ of degree $d$ such that the $p$-adic valuation of $N_{q^m}(f = y^q-y)$ is at most
    \begin{align*}
        \frac{1}{p-1}\sum_{\boldsymbol{t}\in\mathcal{T}_{\{d\}}} S_p(r_{\boldsymbol{t}})=\frac{1}{p-1}S_p(q^m-1)=em=em\left\lceil \frac{k}{d} \right\rceil.    
    \end{align*}
    This shows the existence of such $f$ when $d\ge k$.

    Next, suppose $d<k$. We define $s = \left\lceil \frac{k}{d} \right\rceil$. Since $s \cdot d \ge k$ and $d < k$, we can choose $s$ distinct $d$-element subsets $B_1, \ldots, B_s\subseteq[k]$ such that $\bigcup_{j=1}^s B_j = [k]$. For each $1 \le j \le s$, we construct a tuple $\boldsymbol{t}^{(j)} = (t^{(j)}_1, \ldots, t^{(j)}_k)$ as an indicator vector for the subset $B_j$. Since $|\boldsymbol{t}^{(j)}|=|B_j| = d$, each $\boldsymbol{t}^{(j)}$ lies in $\mathcal{T}_{\{d\}}$. We then define $\boldsymbol{r}$ by
    \begin{align*}
        r_{\boldsymbol{t}}=\begin{cases}
        q^m-1, & \text{if } \boldsymbol{t} = \boldsymbol{t}^{(j)} \text{ for } 1 \le j \le s,\\
        0, & \text{otherwise.}
        \end{cases}
    \end{align*}
    It is easy to verify that such $\boldsymbol{r}$ satisfies \eqref{eq:p_density_program}. Therefore, by Theorem \ref{thm:p_density_program_RM_codes}, there exists a homogeneous polynomial $f$ of degree $d$ whose $p$-adic valuation is bounded above by
    \begin{align*}
        \frac{1}{p-1}\sum_{\boldsymbol{t}\in\mathcal{T}_{\{d\}}} S_p(r_{\boldsymbol{t}}) = \frac{1}{p-1}\sum_{j=1}^s S_p(q^m-1)
        = s \cdot \frac{em(p-1)}{p-1}  = em \left\lceil \frac{k}{d} \right\rceil.
    \end{align*}
    This completes the proof.
\end{proof}

\begin{theorem}
\label{thm:homogeneous_polynomial}
    Let $f\in\mathbb{F}_{q^m}[x_1,\ldots,x_k]$ be a homogeneous polynomial of degree $d\ge 1$. Then the $p$-adic valuation of $N_{q^m}(f=y^q-y)$ is at least  
    \[ \left\lceil  \frac{em}{\left(d,\frac{q^m-1}{q-1}\right)}\left\lceil\frac{\left(d,\frac{q^m-1}{q-1}\right)\cdot k}{d}\right\rceil \right\rceil . \]
    When $\left(d,\frac{q^m-1}{q-1}\right)=1$, for any homogeneous polynomial $f$ of degree $d$, we have 
    \begin{align*}
    \nu_p\left( N_{q^m}(f=y^q-y) \right) \ge em\left\lceil\frac{k}{d}\right\rceil ,    
    \end{align*}
    and there exists such a homogeneous polynomial $f$ achieving this lower bound. 
\end{theorem}

\begin{proof}
    By Theorem \ref{thm:p_density_program_RM_codes}, the $p$-adic valuation of $N_{q^m}(f=y^q-y)$ is at least
    \begin{equation*}
    \begin{aligned}
        \min\quad &V(\boldsymbol{r})=\frac{1}{p-1}\sum_{\boldsymbol{t}\in \mathcal{T}_{\{d\}}}S_p(r_{\boldsymbol{t}}),\\
        \text{s.t.}\quad &\begin{cases}
            \boldsymbol{r}\in [0,q^m-1]^{\mathcal{T}_{\{d\}}},\\
            |\boldsymbol{r}|\equiv 0\pmod{q-1},\\
            \sum_{\boldsymbol{t}\in \mathcal{T}_{\{d\}}}r_{\boldsymbol{t}}t_i \text{ is a positive multiple of } q^m-1,\ 1\le i\le k.
        \end{cases}
    \end{aligned}
    \end{equation*}
    Define the $p$-ary cyclic shift 
    \begin{align*}
    \tau:[0, q^m-1]&\to [0, q^m-1],  \\   
     \sum_{i=0}^{em-1}x_i p^i&\mapsto x_{em-1}+\sum_{i=0}^{em-2}x_i p^{i+1},
    \end{align*}
    where $0\le x_i\le p-1$ for $0\le i\le em-1$. Clearly, $\tau(x)\equiv px\pmod{q^m-1}$. For the tuple $\boldsymbol{r}$, we apply this map component-wise, namely, $\tau(\boldsymbol{r})=(\tau(r_{\boldsymbol{t}}))_{\boldsymbol{t}\in\mathcal{T}_{\{d\}}}$. We have
    \begin{equation}
    \label{eq:homogeneous_lower_bound}
    \begin{aligned}
         V(\boldsymbol{r}) = \frac{1}{p-1}\sum_{\boldsymbol{t}\in \mathcal{T}_{\{d\}}}S_p(r_{\boldsymbol{t}}) = \frac{1}{p^{em}-1}\sum_{\boldsymbol{t}\in \mathcal{T}_{\{d\}}}\sum_{h = 0}^{em-1}\tau^h(r_{\boldsymbol{t}})= \frac{1}{q^m-1}\sum_{h = 0}^{em-1}\sum_{\boldsymbol{t}\in \mathcal{T}_{\{d\}}}\tau^h(r_{\boldsymbol{t}}),
    \end{aligned}
    \end{equation}
    where the second equality follows by the digit sum identity $\sum_{h=0}^{em-1} \tau^h(x) = \frac{p^{em}-1}{p-1} S_p(x)$.

    We next provide a lower bound for the term $\sum_{\boldsymbol{t}\in \mathcal{T}_{\{d\}}}\tau^h(r_{\boldsymbol{t}})$ for $0\le h\le em-1$. For any $1\le i\le k$, since $\sum_{\boldsymbol{t}\in \mathcal{T}_{\{d\}}}r_{\boldsymbol{t}}t_i > 0$ and $\tau(x)=0$ if and only if $x=0$, the sum $\sum_{\boldsymbol{t}\in \mathcal{T}_{\{d\}}} \tau^h(r_{\boldsymbol{t}}) t_i$ is strictly positive. Moreover, this sum is a multiple of $q^m-1$ since
    $$\sum_{\boldsymbol{t}\in \mathcal{T}_{\{d\}}}\tau^h(r_{\boldsymbol{t}})t_i \equiv p^h \sum_{\boldsymbol{t}\in \mathcal{T}_{\{d\}}}r_{\boldsymbol{t}}t_i \equiv 0 \pmod{q^m-1}.$$
    Note that
    \[ d \sum_{\boldsymbol{t}\in \mathcal{T}_{\{d\}}} \tau^h(r_{\boldsymbol{t}}) = \sum_{\boldsymbol{t}\in \mathcal{T}_{\{d\}}} \tau^h(r_{\boldsymbol{t}}) |\boldsymbol{t}| = \sum_{i=1}^k \left( \sum_{\boldsymbol{t}\in \mathcal{T}_{\{d\}}} \tau^h(r_{\boldsymbol{t}}) t_i \right). \]
 Therefore, there exists some $k'\ge k$ such that 
    \[ d \sum_{\boldsymbol{t}\in \mathcal{T}_{\{d\}}} \tau^h(r_{\boldsymbol{t}})= k'(q^m-1) \implies \frac{d}{(d,q^m-1)} \sum_{\boldsymbol{t}\in \mathcal{T}_{\{d\}}} \tau^h(r_{\boldsymbol{t}})= k'\cdot\frac{q^m-1}{(d,q^m-1)}. \]
Since $\left(\frac{d}{(d,q^m-1)}, \frac{q^m-1}{(d,q^m-1)}\right)=1$, the term $\sum_{\boldsymbol{t}\in \mathcal{T}_{\{d\}}} \tau^h(r_{\boldsymbol{t}})$ is a multiple of $\frac{(q^m-1)}{(d,q^m-1)}$. Note also that $\sum_{\boldsymbol{t}\in\mathcal{T}_{\{d\}}}\tau^h(r_{\boldsymbol{t}})$ is a multiple of $q-1$ since
\[\sum_{\boldsymbol{t}\in\mathcal{T}_{\{d\}}}\tau^h(r_{\boldsymbol{t}})\equiv p^h\sum_{\boldsymbol{t}\in\mathcal{T}_{\{d\}}}r_{\boldsymbol{t}}\equiv0\pmod{q-1}.\]
Therefore, $\sum_{\boldsymbol{t}\in \mathcal{T}_{\{d\}}} \tau^h(r_{\boldsymbol{t}})$ is a multiple of $\mathrm{lcm}\left(\frac{q^m-1}{(d,q^m-1)},q-1\right)$, and thus
    \begin{align*}
    \sum_{\boldsymbol{t}\in \mathcal{T}_{\{d\}}} \tau^h(r_{\boldsymbol{t}}) &=\frac{k'(q^m-1)}{\mathrm{lcm}\left(\frac{q^m-1}{(d,q^m-1)},q-1\right)d}\cdot\mathrm{lcm}\left(\frac{q^m-1}{(d,q^m-1)},q-1\right)\\
    &\ge\left\lceil\frac{k(q^m-1)}{\mathrm{lcm}\left(\frac{q^m-1}{(d,q^m-1)},q-1\right)d}\right\rceil\mathrm{lcm}\left(\frac{q^m-1}{(d,q^m-1)},q-1\right). 
    \end{align*}
    By expressing $q-1 = \frac{q^m-1}{(q^m-1)/(q-1)}$ and applying the identity $\mathrm{lcm}\left(\frac{N}{a},\frac{N}{b}\right)=\frac{N}{(a,b)}$, we can simplify the above inequality as 
    \begin{equation}
    \begin{aligned}
    \label{eq:homogeous_lower_bound_2}
        \sum_{\boldsymbol{t}\in \mathcal{T}_{\{d\}}} \tau^h(r_{\boldsymbol{t}})\ge\left\lceil\frac{\left(d,\frac{q^m-1}{q-1}\right)k}{d}\right\rceil\frac{q^m-1}{\left(d,\frac{q^m-1}{q-1}\right)}.  
    \end{aligned}
    \end{equation}
    The first assertion follows from \eqref{eq:homogeneous_lower_bound} and \eqref{eq:homogeous_lower_bound_2}, while the existence result in the second assertion is a direct consequence of Proposition \ref{upper bound}.
\end{proof}

\begin{remark}
In fact, we can relax the condition $\left(d,\frac{q^m-1}{q-1}\right)=1$ to be $\left\lceil\frac{\left(d,\frac{q^m-1}{q-1}\right)\cdot k}{d}\right\rceil=\left(d,\frac{q^m-1}{q-1}\right)\left\lceil\frac{k}{d}\right\rceil$ such that the conclusion holds. Indeed, there are some examples such that $\left\lceil\frac{\left(d,\frac{q^m-1}{q-1}\right)\cdot k}{d}\right\rceil=\left(d,\frac{q^m-1}{q-1}\right)\left\lceil\frac{k}{d}\right\rceil$ but $\left(d,\frac{q^m-1}{q-1}\right)\ne1$, e.g., $q=3$, $m=4$, $k=14$ and $d=15$.    
\end{remark}

Below, we consider the $p$-adic valuation of $N_{q^m}(f = y^q-y)$ for  arbitrary, not necessarily homogeneous polynomials $f$.

\begin{theorem}
\label{thm:non_homogeneous_first_bound}
    Let $f\in\mathbb{F}_{q^m}[x_1,\ldots,x_k]$ be a polynomial of degree $d\ge 1$. Then the $p$-adic valuation of $N_{q^m}(f=y^q-y)$ is at least  
    \[ \left\lceil \frac{em(q-1)}{q^m-1} \left\lceil \frac{k(q^m-1)}{d(q-1)} \right\rceil \right\rceil. \]
\end{theorem}

The proof of Theorem \ref{thm:non_homogeneous_first_bound} is similar to that of Theorem \ref{thm:homogeneous_polynomial} and is deferred to  Appendix \ref{appendix_B}.

This result actually recovers the famous Ax's theorem, which characterizes the $p$-adic valuation of $N_{q^m}(f = 0)$ for arbitrary $f$ with degree $d$.

\begin{corollary}[\cite{ax1964zeroes}, Ax's theorem]\label{Ax_theorem}
    Let $f\in \mathbb{F}_q[x_1,\ldots,x_k]$ be a polynomial of degree $d$. Then the number of solution of $f =  0$ in $\mathbb{F}_q^k$ is divisible by $q^{\lceil k/d\rceil -1}$. Furthermore, there exists a polynomial $f$ with degree $d$ such that $N_{q}(f = 0)$ is sharply divided by $q^{\lceil k/d\rceil -1}$.
\end{corollary}

\begin{proof}
Note that $N_{q}(f = y^q-y) = qN_{q}(f = 0)$. We set $m = 1$ in Theorem \ref{thm:homogeneous_polynomial} and Theorem \ref{thm:non_homogeneous_first_bound}. Theorem \ref{thm:non_homogeneous_first_bound} proves the lower bound and Theorem \ref{thm:homogeneous_polynomial} guarantees the existence of such polynomial.
\end{proof}

We first define  
\[W_q(D,\mathbb{F}_{q^m}^k) = \max\{S_q(\boldsymbol{t})=S_q(t_1)+\cdots+S_q(t_k) :\boldsymbol{t}\in [0,q^m-1]^k, |\boldsymbol{t}|\in D \},\]
where $S_q(t)$ is the digit sum of $t$ in the $q$-adic expansion. Next, we recall Mueller's result on the $p$-divisibility of the points on Artin--Shreier curves.

\begin{theorem}[{\cite[Theorem VI.3]{Mueller2013}}]
Let $f(x)\in \mathbb{F}_q[x]$ be a polynomial with degree set $D$. The $p$-adic valuation of $N_{q^m}(f(x) = y^q-y)$ is at least   
\[e\left\lceil \frac{m}{W_q(D,\mathbb{F}_{q^m})} \right\rceil.\]
\end{theorem}

In fact, Ax's theorem implies another lower bound (Theorem \ref{thm:non_homogeneous_second_bound}) for $\nu_p(N_{q^m}(f = y^q-y))$ for general $f\in \mathbb{F}_{q^m}[x_1,\ldots,x_k]$. Note that $N_{q^m}(f = y^q-y) = qN_{q^m}(\mathrm{Tr}_{q^m/q}(f)=0)$. The idea is to treat the functions $\mathrm{Tr}_{q^m/q}(x_1^{t_1}\cdots x_k^{t_k}):\mathbb{F}_{q^m}^k\to \mathbb{F}_q$ as a $km$-variate homogeneous degree-$S_q(\boldsymbol{t})$ polynomial function mapping from $\mathbb{F}_q^{mk}$ to $\mathbb{F}_q$. Thus, we can generalize Mueller's result by showing that the divisibility result indeed holds for any multivariate polynomials in the extension field $\mathbb{F}_{q^m}$. The proof of the subsequent theorem is deferred to Appendix \ref{appendix_B}.

\begin{theorem}
\label{thm:non_homogeneous_second_bound}
    Let $f\in\mathbb{F}_{q^m}[x_1,\ldots,x_k]$ be a non-constant polynomial with degree set $D$. Then the $p$-adic valuation of $N_{q^m}(f=y^q-y)$ is at least
    \[ e \left\lceil \frac{mk}{W_q(D,\mathbb{F}_{q^m}^k)} \right\rceil, \]
    where $W_q(D,\mathbb{F}_{q^m}^k) = \max\{S_q(\boldsymbol{t})=S_q(t_1)+\cdots+S_q(t_k) :\boldsymbol{t}\in [0,q^m-1]^k, |\boldsymbol{t}|\in D \}$.
\end{theorem}

Moreover, we can advance the reduction in the proof of Theorem \ref{thm:non_homogeneous_second_bound} by descending to the prime field $\mathbb{F}_p$. Treating $\mathbb{F}_{q^m}$ as an $em$-dimensional vector space over $\mathbb{F}_p$ allows the condition $\mathrm{Tr}_{q^m/q}(f)=0$ in $\mathbb{F}_q$ to decompose into a system of $e$ simultaneous polynomial equations over $\mathbb{F}_p$ in $emk$ variables. An application of the classical Ax-Katz theorem \cite{Katz1971on} to this system yields another lower bound. The proof of the following theorem is deferred to Appendix \ref{appendix_B}.

\begin{theorem}
\label{thm:non_homogeneous_third_bound}
    Let $q=p^e$ and $f\in\mathbb{F}_{q^m}[x_1,\ldots,x_k]$ be a non-constant polynomial with degree set $D$. Then the $p$-adic valuation of $N_{q^m}(f=y^q-y)$ is at least  
    \[ \left\lceil \frac{emk}{W_{p}(D,\mathbb{F}_{q^m}^k)} \right\rceil, \]
    where $W_{p}(D,\mathbb{F}_{q^m}^k) = \max\{S_p(\boldsymbol{t}) = S_p(t_1)+\cdots+S_p(t_k) : \boldsymbol{t} \in [0,q^m-1]^k, |\boldsymbol{t}|\in D\}$.
\end{theorem}

\section{Conclusions}\label{sec:conclusions}

Our approach completely circumvents the weight polarization method originally introduced by Ward in \cite{ward1979combinatorial}, which is later applied to establish Ward's divisibility criterion \cite[Theorem 5.3]{ward1990weight}. Instead, we adopt a more standard algebraic framework to analyze the weights of codewords: expressing the weights as exponential sums and subsequently applying Stickelberger's theorem. This technique dates back to the proof of Ax's theorem on the zeros of polynomials \cite{ax1964zeroes}, and has underpinned subsequent works by Delsarte, McEliece \cite{Delsarte1971, McEliece1972, delsarte1976} on the $p$-adic valuation of linear codes. 

This framework provides an effective pathway to establish lower bounds on $p$-adic valuations. Indeed, recall from \eqref{eq:weight_formula_2} in the proof of Theorem~\ref{thm:trace_version_criterion_general_case} that the weight of any codeword $c(\boldsymbol{\alpha})\in C$ can be expressed as
\[
\mathrm{wt}(c(\boldsymbol{\alpha})) = -(q-1)\sum_{\boldsymbol{r}\in \mathcal{R} }\frac{1}{q}\left(\prod_{i = 1}^k\lambda_{r_i}\right) \left(\prod_{i =1 }^kT(\alpha_i)^{r_i}\right) \left(\sum_{j = 1}^n\prod_{i = 1}^kT(g_{ij})^{r_i}\right).
\]
By Stickelberger's theorem (Theorem \ref{Stickelberger's Theorem}), the $p$-adic valuation of each summand above is bounded from below by
\[
\Theta = \min_{\boldsymbol{r}\in\mathcal R}
\left\{
\frac{1}{p-1}\sum_{i=1}^k S_p(r_i)
+
\nu_p\left(
\sum_{j=1}^n\prod_{i=1}^kT(g_{ij})^{r_i}
\right)
-e
\right\}.
\]
Consequently, we have $\nu_p(\mathrm{wt}(c(\boldsymbol{\alpha}))) \ge \Theta$, which implies that $\nu_p(C)\ge \Theta$.

To establish the sharpness of such lower bounds, previous works of Delsarte, McEliece often relied on highly intricate calculations. In contrast, we first develop a divisibility criterion for trace codes and then apply it to the trace representation of semisimple abelian codes, rather than working directly with the group-algebra structure. We believe that the same approach can also yield sharp divisibility results for abelian codes over other rings.

\begin{appendices}
\section{Proofs of Theorems \ref{thm:non_homogeneous_first_bound} and \ref{thm:non_homogeneous_second_bound}--\ref{thm:non_homogeneous_third_bound}}
\label{appendix_B}

\begin{proof}[Proof of Theorem \ref{thm:non_homogeneous_first_bound}]
    Let $D = [0,d]$. By Theorem \ref{thm:p_density_program_RM_codes}, the $p$-adic valuation of $N_{q^m}(f=y^q-y)$ is at least
    \begin{equation*}
    \label{eq:p_density_program_arbitrary}
    \begin{aligned}
        \min\quad &V(\boldsymbol{r})=\frac{1}{p-1}\sum_{\boldsymbol{t}\in \mathcal{T}_{D}}S_p(r_{\boldsymbol{t}}),\\
        \text{s.t.}\quad &\begin{cases}
            \boldsymbol{r}\in [0,q^m-1]^{\mathcal{T}_{D}},\\
            |\boldsymbol{r}|\equiv 0\pmod{q-1},\\
            \sum_{\boldsymbol{t}\in \mathcal{T}_{D}}r_{\boldsymbol{t}}t_i \text{ is a positive multiple of } q^m-1,\ 1\le i\le k.
        \end{cases}
    \end{aligned}
    \end{equation*}
    Then we have
    \begin{equation}
    \label{eq:arbitrary_lower_bound}
    \begin{aligned}
         V(\boldsymbol{r}) = \frac{1}{p-1}\sum_{\boldsymbol{t}\in \mathcal{T}_{D}}S_p(r_{\boldsymbol{t}}) = \frac{1}{p^{em}-1}\sum_{\boldsymbol{t}\in \mathcal{T}_{D}}\sum_{h = 0}^{em-1}\tau^h(r_{\boldsymbol{t}})= \frac{1}{q^m-1}\sum_{h = 0}^{em-1}\sum_{\boldsymbol{t}\in \mathcal{T}_{D}}\tau^h(r_{\boldsymbol{t}}).
    \end{aligned}
    \end{equation}

    Let $0\le h\le em-1$. Since $|\boldsymbol{t}| \le d$ for all $\boldsymbol{t} \in \mathcal{T}_{D}$, we have the inequality
    \[ d \sum_{\boldsymbol{t}\in \mathcal{T}_{D}} \tau^h(r_{\boldsymbol{t}}) \ge \sum_{\boldsymbol{t}\in \mathcal{T}_{D}} \tau^h(r_{\boldsymbol{t}}) |\boldsymbol{t}| = \sum_{i=1}^k \left( \sum_{\boldsymbol{t}\in \mathcal{T}_{D}} \tau^h(r_{\boldsymbol{t}}) t_i \right). \]
    Similar to the homogeneous case, for each coordinate $i$, the inner sum 
    \[ \sum_{\boldsymbol{t}\in \mathcal{T}_{D}}\tau^h(r_{\boldsymbol{t}})t_i \equiv p^h \sum_{\boldsymbol{t}\in \mathcal{T}_{D}}r_{\boldsymbol{t}}t_i \equiv 0 \pmod{q^m-1}\]
    is a positive multiple of $q^m-1$. Therefore,
    \[ d \sum_{\boldsymbol{t}\in \mathcal{T}_{D}} \tau^h(r_{\boldsymbol{t}}) \ge k(q^m-1) \implies \sum_{\boldsymbol{t}\in \mathcal{T}_{D}} \tau^h(r_{\boldsymbol{t}}) \ge \frac{k(q^m-1)}{d}. \]
    Note also that $\sum_{\boldsymbol{t}\in\mathcal{T}_{D}}\tau^h(r_{\boldsymbol{t}})\equiv\sum_{\boldsymbol{t}\in\mathcal{T}_{D}}p^hr_{\boldsymbol{t}}\equiv 0\pmod{q-1}$. Hence $\sum_{\boldsymbol{t}\in\mathcal{T}_{D}}\tau^h(r_{\boldsymbol{t}})$ must be a positive multiple of $q-1$. This yields 
    \begin{equation*}
    \label{eq:arbitrary_lower_bound_2}
    \begin{aligned}
        \sum_{\boldsymbol{t}\in \mathcal{T}_{D}} \tau^h(r_{\boldsymbol{t}})\ge (q-1)\left\lceil\frac{k(q^m-1)}{d(q-1)}\right\rceil.  
    \end{aligned}
    \end{equation*}
    Summing this inequality over all $0 \le h \le em-1$, and substituting it back into \eqref{eq:arbitrary_lower_bound}, we obtain
    \begin{align*}
        V(\boldsymbol{r}) \ge \frac{1}{q^m-1} \sum_{h = 0}^{em-1} (q-1)\left\lceil\frac{k(q^m-1)}{d(q-1)}\right\rceil = \frac{em(q-1)}{q^m-1} \left\lceil\frac{k(q^m-1)}{d(q-1)}\right\rceil.
    \end{align*}
    This completes the proof.
\end{proof}

\begin{proof}[Proof of Theorem \ref{thm:non_homogeneous_second_bound}]
    Let $\{\gamma_1,\ldots,\gamma_m\}$ be a basis of $\mathbb{F}_{q^m}$ over $\mathbb{F}_q$. Each element $x_i \in \mathbb{F}_{q^m}$ is an $\mathbb{F}_q$-linear combination $x_i = \sum_{j=1}^m y_{i,j}\gamma_j$ for some $y_{i,j}\in \mathbb{F}_q$. Consider an arbitrary monomial $x_1^{t_1}\cdots x_k^{t_k}$ with $|\boldsymbol{t}|\in D$. We write the $q$-adic expansion of each exponent as $t_i = \sum_{\ell=0}^{m-1} t_{i,\ell} q^\ell$ for $0\le t_{i,\ell}\le q-1$. By the property $y_{i,j}^q = y_{i,j}$ for $y_{i,j} \in \mathbb{F}_q$, we can expand the monomial in variables $y_{i,j}$ as
    \begin{align*}
        x_1^{t_1}\cdots x_k^{t_k} =\prod_{i =1}^k \prod_{\ell = 0}^{m-1} x_i^{t_{i,\ell}q^\ell}= \prod_{i =1}^k \prod_{\ell = 0}^{m-1} \left( \sum_{j=1}^m y_{i,j}\gamma_j \right)^{t_{i,\ell}q^\ell} = \prod_{i =1}^k \prod_{\ell = 0}^{m-1} \left( \sum_{j=1}^m y_{i,j}\gamma_j^{q^\ell} \right)^{t_{i,\ell}}.
    \end{align*}
    It is easy to see that this expansion forms a homogeneous polynomial in $mk$ variables $y_{i,j}$ over $\mathbb{F}_{q^m}$ with total degree exactly $\sum_{i=1}^k S_q(t_i) = S_q(\boldsymbol{t})$. Therefore, the function $\mathrm{Tr}_{q^m/q}(x_1^{t_1}\cdots x_k^{t_k})$ is either identically 0, or can be written as a homogeneous polynomial in $\mathbb{F}_q[y_{1,1}, \ldots, y_{k,m}]$ with degree $S_q(\boldsymbol{t})$. 
    
    Consequently, for any $f \in \mathbb{F}_{q^m}[x_1, \ldots, x_k]$ of degree  $d$, its trace form $\mathrm{Tr}_{q^m/q}(f(x_1,\ldots,x_k))$ can be expressed as a polynomial $F(y_{1,1}, \ldots, y_{k,m}) \in \mathbb{F}_q[y_{1,1}, \ldots, y_{k,m}]$, which is either the zero polynomial, or of degree at most $\max_{|\boldsymbol{t}| \in D} S_q(\boldsymbol{t}) = W_q(D,\mathbb{F}_{q^m}^k)$.
    
    If $F$ is not the zero polynomial, then by Corollary \ref{Ax_theorem} (Ax's theorem), we have
    \[\nu_p(N_{q^m}(\mathrm{Tr}_{q^m/q}(f)=0))=\nu_p(N_{q^m}(F=0))\ge e\left( \left\lceil \frac{mk}{W_q(D,\mathbb{F}_{q^m}^k)} \right\rceil - 1 \right).\]
    If $F$ is the zero polynomial, then the above inequality also holds. By \eqref{solutions equivalence}, the $p$-adic valuation of $N_{q^m}(f=y^q-y)$ satisfies
    \begin{align*}
        \nu_p(N_{q^m}(f=y^q-y)) = \nu_p(q) + \nu_p(N_{q^m}(\mathrm{Tr}_{q^m/q}(f)=0)) 
        \ge  e \left\lceil \frac{mk}{W_q(D,\mathbb{F}_{q^m}^k)} \right\rceil.
    \end{align*}
    This completes the proof.
\end{proof}

\begin{proof}[Proof of Theorem \ref{thm:non_homogeneous_third_bound}]
    Let $\{\gamma_1,\ldots,\gamma_{em}\}$ be a basis of $\mathbb{F}_{q^m}$ over $\mathbb{F}_p$. Each element $x_i \in \mathbb{F}_{q^m}$ is an $\mathbb{F}_p$-linear combination $x_i = \sum_{j=1}^{em} y_{i,j}\gamma_j$ for some $y_{i,j}\in \mathbb{F}_p$. Consider an arbitrary monomial $x_1^{t_1}\cdots x_k^{t_k}$ with $|\boldsymbol{t}|\in D$. We write the base-$p$ expansion of each exponent as $t_i = \sum_{\ell=0}^{em-1} t_{i,\ell} p^\ell$ for $0\le t_{i,\ell}\le p-1$. By the property $y_{i,j}^p = y_{i,j}$ for $y_{i,j} \in \mathbb{F}_p$, we can expand the monomial in variables $y_{i,j}$ as
    \begin{align*}
        x_1^{t_1}\cdots x_k^{t_k} =\prod_{i =1}^k \prod_{\ell = 0}^{em-1} x_i^{t_{i,\ell}p^\ell}= \prod_{i =1}^k \prod_{\ell = 0}^{em-1} \left( \sum_{j=1}^{em} y_{i,j}\gamma_j \right)^{t_{i,\ell}p^\ell} = \prod_{i =1}^k \prod_{\ell = 0}^{em-1} \left( \sum_{j=1}^{em} y_{i,j}\gamma_j^{p^\ell} \right)^{t_{i,\ell}}.
    \end{align*}
    It is easy to see that this expansion forms a homogeneous polynomial in $emk$ variables $y_{i,j}$ over $\mathbb{F}_{q^m}$ with degree exactly $\sum_{i=1}^k S_p(t_i) = S_p(\boldsymbol{t})$. Therefore, any polynomial $f \in \mathbb{F}_{q^m}[x_1, \ldots, x_k]$ with degree set $D$ can be expressed as a polynomial $G(y_{1,1}, \ldots, y_{k,em})$ over $\mathbb{F}_{q^m}$ with degree at most $\max_{|\boldsymbol{t}| \in D} S_p(\boldsymbol{t}) = W_p(D,\mathbb{F}_{q^m}^k)$.
    
    To evaluate $N_{q^m}(\mathrm{Tr}_{q^m/q}(f)=0)$, let $\{\beta_1, \ldots, \beta_e\}$ be a basis of $\mathbb{F}_q$ over $\mathbb{F}_p$. We define
    \[G_r(y_{1,1}, \ldots, y_{k,em}) = \mathrm{Tr}_{q/p}(\beta_r\mathrm{Tr}_{q^m/q}(f(x_1,\ldots,x_k))),\quad \text{for }1\le r\le e,\]
    where each $G_r \in \mathbb{F}_p[y_{1,1}, \ldots, y_{k,em}]$ is either the zero polynomial, or a polynomial over $\mathbb{F}_p$ with degree at most $\deg(G) \le W_p(D,\mathbb{F}_{q^m}^k)$. The equation $\mathrm{Tr}_{q^m/q}(f)=0$ is thus equivalent to the system of $e$ equations $G_1 = 0, \ldots, G_e = 0$ over $\mathbb{F}_p$. 

    If no $G_i$ is the zero polynomial, then by the Ax-Katz theorem \cite{Katz1971on}, the number of common solutions $(y_{1,1}, \ldots, y_{k,em}) \in \mathbb{F}_p^{emk}$ to this system is divisible by $p^{s}$, where the exponent $s$ is bounded below by
    \[ s \ge \left\lceil \frac{emk - \sum_{r=1}^e \deg(G_r)}{\max_{1\le r\le e} \deg(G_r)} \right\rceil \ge \left\lceil \frac{emk - eW_p(D,\mathbb{F}_{q^m}^k)}{W_p(D,\mathbb{F}_{q^m}^k)} \right\rceil = \left\lceil \frac{emk}{W_p(D,\mathbb{F}_{q^m}^k)} \right\rceil - e. \]
    If some $G_i$'s but not all are the zero polynomial, then the inequality $s \ge \left\lceil \frac{emk}{W_p(D,\mathbb{F}_{q^m}^k)} \right\rceil - e$ also holds by dropping those zero polynomials. If all $G_i$'s are the zero polynomial, then $s \ge \left\lceil \frac{emk}{W_p(D,\mathbb{F}_{q^m}^k)} \right\rceil - e$ is clearly true. Since $N_{q^m}(\mathrm{Tr}_{q^m/q}(f)=0)$ is exactly the number of common zeros of $G_1, \ldots, G_e$, its $p$-adic valuation is at least $s$. By \eqref{solutions equivalence}, the $p$-adic valuation of $N_{q^m}(f=y^q-y)$ satisfies
    \begin{align*}
        \nu_p(N_{q^m}(f=y^q-y)) &= \nu_p(q) + \nu_p(N_{q^m}(\mathrm{Tr}_{q^m/q}(f)=0))\\
        &\ge e + \left( \left\lceil \frac{emk}{W_p(D,\mathbb{F}_{q^m}^k)} \right\rceil - e \right) = \left\lceil \frac{emk}{W_p(D,\mathbb{F}_{q^m}^k)} \right\rceil.
    \end{align*}
    This completes the proof.
\end{proof}

\end{appendices}

\section*{Acknowledgment}

We thank Tao Feng from Zhejiang University for his valuable suggestions. Hexiang Huang and Sihuang Hu are supported in part by the National Key Research and Development Program of China under Grant 2021YFA1001000, in part by the National Natural Science Foundation of China under Grant 12571576 and Grant 12231014, and in part by the Shandong Provincial Natural Science Foundation under Grant ZR2025QA05. Haihua Deng is partially supported by the National Key R\&D Program of China under Grant No.~2025YFA1017700, the National Natural Science Foundation of China under Grant No.~123B2011, and the Postdoctoral Fellowship Program and China Postdoctoral Science Foundation under Grant No.~BX20250059.

\end{document}